\documentclass{article}

\usepackage[utf8]{inputenc}
\usepackage[english]{babel}
\usepackage{amsmath}
\usepackage{amsfonts}
\usepackage{amssymb}
\usepackage{amsthm}
\usepackage{graphicx}

\usepackage[a4paper, left=2.5cm, right=2.5cm, top=2cm, bottom=2cm]{geometry}
\usepackage{xcolor}
\usepackage{hyperref}
\usepackage{cite}
\usepackage{subcaption}

\newtheorem{theorem}{Theorem}

\newtheorem{statement}{Statement}
\newtheorem{remark}{Remark}
\newtheorem{corollary}{Corollary}
\newtheorem{definition}{Definition}

\usepackage{stmaryrd}

\newcommand{\bu}{{\bf u}}
\newcommand{\bp}{{ p}}
\newcommand{\bg}{{\bf g}}
\newcommand{\bfv}{{\bf f}}
\newcommand{\bFi}{{\bf \Phi}}
\newcommand{\bv}{{\bf v}}
\newcommand{\bn}{{\bf n}}
\newcommand{\bchi}{\boldsymbol{ \chi}}
\newcommand{\bpsi}{\boldsymbol{ \psi}}
\newcommand{\mC}{\mathcal{C}}

\newcommand{\bsigmaE}{\boldsymbol{ \sigma}^\varepsilon}
\newcommand{\bsigma}{\boldsymbol{ \sigma}}

\title{Homogenization procedure in the Cauchy problem for the elastic composites. \footnote{
S.S. acknowledge the financial support of FAPERJ APQ1 processo №.  E-26/210.614/2024 and of the Coordenação de Aperfeiçoamento de Pessoal de Nível Superior – Brasil (CAPES) – Finance Code 001.

R. R.-R. thanks to Chamada CNPq №. 09/2023 PQ-2 Productividade em Pesquisa, processo №. 307188/2023-0 e ao Edital UFF PROPPI N. 05/2022.
}}
\author{Sergey Sergeev\footnote{Instituto de Matem\'{a}tica, Universidade Federal do Rio de Janeiro, Rio de Janeiro, Brazil,  Av. Athos da Silveira Ramos, 149, \texttt{sergeevse1@im.ufrj.br}}, Reinaldo Rodriguez-Ramos\footnote{Facultad de Matemática y Computación, Universidad de La Habana, La Habana, Cuba; PPG-MCCT, \\
Universidade Federal Fluminense, Volta Redonda, Brazil, \texttt{rerora2006@gmail.com}}}
\begin{document}

\maketitle

\abstract{In the present work we consider the problem of the homogenization of the Cauchy problem for the equations of the elasticity. We assume the composite material is constructed as a periodic repetition along the chosen  axis of the unit cell which models to different materials. The homogenization procedure is based on the operator separation of variables, which  gives the new incites for the homogenization of the elasticity equations and provides new way of the homogenization and constructing the homogenized equation. 

}

\section{Introduction}
\label{sec_intro}

The homogenization problem in the elastic problems is one of the classical problems \cite{oleinik_mathematical_1992, bakhvalov_homogenisation_1989, bensoussan_asymptotic_2011, cioranescu_introduction_1999, sanchez-palencia_non-homogeneous_1980, penta_asymptotic_2017, parnell_homogenization_2008}. 
Following \cite{brito-santana_dispersive_2015}, in the present work   we consider the composite elastic material in the whole space with periodic layer structure  corresponding to the unit cell 
$$
Y^\varepsilon=\{x_2\in \mathbb{R},\,x_1\in[0,\,\varepsilon l_1]\},\quad y_1=\frac{x_1}{\varepsilon}\in[0,\,l_1].
$$
The composite material is modeled by the periodic repetition of the $Y^\varepsilon$ along the axis $x_1$. The unit cell reflects the combination of two different materials, called A and B. We introduce the coefficient $\gamma$ - the volume fraction - which  describes the width of each  material in the unit cell
\begin{equation}
\label{two_mater}
y_1\in[0,\,\gamma l_1) \quad y_1\in(\gamma l_1,\,l_1],\quad \gamma\in(0,\,1).
\end{equation}
The $l_1$ is the proper length of cell and parameter $\varepsilon\ll 1$ is a small parameter which is the ratio of length $l_1$ to the length $L$ representing the distance of the wave propagation up to the given fixed time moment $T$.

The homogenization problem is formulated as the consideration of the limit $\varepsilon\to 0$. In \cite{brito-santana_dispersive_2015} the harmonic dependence on time was assumed which leads to the spectral problem. The result of the homogenization was the dispersion relation obtained up to the $O(\varepsilon^2)$. In the present problem we assume the presence of the localized initial data  and we cannot assume the harmonic dependence on time.

We propose the  different approach to the homogenization procedure, which is based on the adiabatic approximation (separation of variables in operator form).  In general one cannot exactly separate slow and fast variables, but it can be done approximately. This method of separation of variables proved to be useful in the homogenization problems of the theory of linear waves propagation, see for example \cite{grushin_homogenization_2013, dobrokhotov_asymptotic_2016, sergeev_asymptotic_2022}. The general approach is based on the representation of the solution in the form of the action of the (pseudo)-differential  operator \cite{maslov_semi-classical_2001, hormander_analysis_2007, martinez_introduction_2002, zworski_semiclassical_2012} on some other function. The main idea is that the only operator depends on the fast variables, while the function satisfies the equation with constant coefficients.

 The similar approach was proposed in \cite{allaire_crime_2022}, where the solution was approximated with the help of sequence of the differential operators in spatial and temporal variables. 

The adiabatic approximation method allows to describe the homogenized equation  and the internal structure of its solution.  Usually the final structure of the homogenized equation is known, but, on the same time, the structure of the solution is not. Our approach provides an alternative way of the analysis and gives very useful internal information about the solution of the homogenized problem. This approach can be extended on the quasi-periodic structure, contrary to \cite{brito-santana_dispersive_2015}.

In the present work we develop the operator separation of variables for the Cauchy problem for the elastic problems. We do not focus ourselves to the questions of the existence of solutions. We provide a formal asymptotic expansion, the so-called the formal asymptotic. 

This work is organized as follows. In the Section \ref{sec_Prob_stat} we state the problem and formulate the main results. In the Section \ref{sec_num_examp} we provide numerical calculations which illustrate our results, while in the Section \ref{conclus} we provide discussion of the results. Section \ref{sec_Sep_Var} describes the general approach of the operator separation of variable method. In the Section \ref{sec_hom_thm_proof} we provide an asymptotic expansion and proof of the main homogenization theorem. In the Appendix \ref{PsiDO} we describe some useful relations for the pseudo-differential operators, used in the present work.

%
%S.S. acknowledge the financial support of FAPERJ APQ1 processo №.  E-26/210.614/2024 and of the Coordenação de Aperfeiçoamento de Pessoal de Nível Superior – Brasil (CAPES) – Finance Code 001.
%
%R. R.-R. thanks to Chamada CNPq №. 09/2023 PQ-2 Productividade em Pesquisa, processo №. 307188/2023-0 e ao Edital UFF PROPPI N. 05/2022.

\section{Problem statement and the main results}
\label{sec_Prob_stat}
We pose the following Cauchy problem
\begin{gather}
\label{eq1}
\nabla_x\cdot \bsigmaE(x,\,t)-\rho(\frac{x_1}{\varepsilon})\frac{\partial^2}{\partial t^2}\bu^\varepsilon(x,\,t)=0,\quad x=(x_1,\,x_2)\in\mathbb{R}^2,\\
\label{init_cond}
\bu^\varepsilon(x,\,0)=\bfv(x),\quad \frac{\partial \bu^\varepsilon}{\partial t}(x,\,0)=\bg(x).
\end{gather}
Functions $\bfv(x)$ and $\bg(x)$ are the given functions.
We assume the following $\bu^\varepsilon(x,\,t)\in L_{\infty}(\mathbb{R}^2)\cap H^1(\mathbb{R}^2)$, for any time moment $t\ge 0$.

Here the $\bu^\varepsilon(x,\,t)=(u_1(x,\,t),\,u_2(x,\,t))$ is the dislocation vector and
$$
\bsigmaE(x,\,t)=\frac{1}{2}C(\frac{x_1}{\varepsilon}): \left(\nabla\otimes \bu^\varepsilon(x,\,t)+(\nabla\otimes \bu^\varepsilon(x,\,t))^T\right),
$$
 is the stress tensor.
 
 We use the following notation. The symbol ``$\cdot$'' means contraction of the tensors, while ``$:$'' is the double contraction. The symbol ``$\otimes$'' means the Kronecker product.

The tensor $C_{ijkl}(y_1)$ describes the properties of the material and it is the completely positive tensor and satisfies the major and minor symmetry conditions
\begin{equation}
\label{C_tens}
C_{ijkl}(y_1)=C_{jikl}(y_1)=C_{ijlk}(y_1)=C_{klij}(y_1),\quad 
C_{ijkl}\xi_{ij}\xi_{kl}\ge A |\xi_{mn}|^2,\,A>0.
\end{equation}
In our problem this tensor depends only on $y_1$ and it is piece-wise periodic
\begin{equation}
\label{C_tens_period}
C(y_1)=C^A(y_1),\, y_1\in[0,\,\gamma l_1],\quad 
C(y_1)=C^B(y_1),\, y_1\in[\gamma l_1,\,l_1].
\end{equation}
Function $\rho(y_1)$ describes the density of the material and it satisfies the similar conditions
$$
\rho(y_1)=\rho^B(y_1),\, y_1\in[0,\,\gamma l_1],\quad 
\rho(y_1)=\rho^B(y_1),\, y_1\in[\gamma l_1,\,l_1].
$$
Here indices $A$ and $B$ denote the different materials $A$ and $B$.

During the presence of two materials in the unit cell (\ref{two_mater}) we introduce additional conditions on the interface $y_1=\gamma l_1$. Let us denote the set of the points as
$$
S^\varepsilon=\bigcup_{n\in\mathbb{Z}}\left\{x=(x_1,x_2)\in\mathbb{R}^2:\,x_1=\varepsilon(n+\gamma)l_1\right\}.
$$
In this definition, $x_2\in\mathbb{R}$ remains free, so each component of $S^\varepsilon$ is a vertical interface in the physical plane.

We have two following additional conditions in the interface 
\begin{equation}
\label{BC_interface}
\llbracket \bu^\varepsilon(x,\,t) \rrbracket=0,\quad \llbracket \bsigmaE(x,\,t)\cdot \bn\rrbracket=0,\quad x\in S^\varepsilon.
\end{equation} 
Here $\bn$ is the outward normal vector along the positive direction of $x_1$. The brackets $\llbracket \cdot \rrbracket$ define the jumps in the interface points.

We pose the homogenization problem for the \eqref{eq1}, \eqref{init_cond}, \eqref{BC_interface} when the $\varepsilon\to 0$. We provide the solution by the method of the operator separation of variables. The short description of this method is the following. We start from the ansatz of the solution in the form where slow $x$ and rapid $y_1=x_1/\varepsilon$ variables are separated $\bu^\varepsilon(x,\,t)=\bFi(x,\,x_1/\varepsilon,\,t)$. The next step is to write the representation $\bFi(x,\,y_1,\,t)=\bchi(x,\,y_1,\,-i \nabla_x) v(x,\,t)$, where $\bchi(x,\,y_1,\,-i\nabla_x)$ is the vector of unknown pseudo-differential operators and $v(x,\,t)$ is the scalar function. Function $v(x,\,t)$ is the solution of the equation with the constant coefficients. The vector-symbol $\bchi(x,\,y_1,\,p)$ is the solution of some spectral problem with respect to the  variable $y_1$. It is almost impossible to determine the vector-symbol $\bchi(x,\,y_1,\,p)$ exactly, thus we are using the asymptotic expansion of this symbol with respect to the small parameter $\varepsilon$. This leads to the asymptotic approximation to the solution of the original equation \eqref{eq1}.

Now let us present the homogenization results. We start from the definition of the homogenized tensor.
Let us introduce the tensors $C^1(y_1)$ and $C_1(y_1)$ of the 3-rd order and the symmetric matrix $\mC(y_1)$
\begin{gather}
\label{tens_C1_Cup1}
C^1(y_1)=C_{i1kl}(y_1),\quad C_1(y_1)=C_{ijk1}(y_1),\\
\label{matr_C}
\mC(y_1)=\{C_{i1k1}(y_1)\}=\begin{pmatrix}
C_{1111}(y_1) & C_{1121}(y_1)\\
C_{2111}(y_1) & C_{2121}(y_1)
\end{pmatrix}.
\end{gather}

For the matrix $\mC$ and the tensors $C_1$ and $C^1$ we have the following symmetry property.
\begin{statement}
\label{st_C1_Cup1_symmetry}
The matrix \eqref{matr_C} is positive definite
\begin{equation}
\label{matr_C_posit}
\mC(y_1)w_i w_k\ge A |w|^2,\,A>0,\quad \forall w\in\mathbb{R}^2.
\end{equation}

The tensors $C_1(y_1)$ and $C^1(y_1)$ have minor right and left symmetries correspondingly
\begin{equation}
\label{C1_Cup1_symmetry}
(C_1(y_1))_{psq}=(C_1(y_1))_{pqs},\quad (C^1(y_1))_{sqp}=(C^1(y_1))_{spq}.
\end{equation}

The tensors $C_1(y_1)$ and $C^1(y_1)$ are symmetrical in the sense of the cyclic permutation of the indices
\begin{equation}
\label{C1_T_permut}
(C_1(y_1))_{psq}=(C^1(y_1))_{qps}.
\end{equation}
\end{statement}

\begin{proof}
To show positive definiteness of the matrix $\mC$, we use positive definiteness of the tensor \eqref{C_tens}, which is $C_{ijkl}\xi_{ij}\xi_{kl}\ge A |\xi_{mn}|^2$. Let us take $\xi_{i2}=0$, for $i=1,\,2$. Then for this particular case we have $C_{ijkl}\xi_{ij}\xi_{kl}=\mC_{ik}w_i w_k$, where $w_i=(\xi_{11},\,\xi_{21})$. Because $(\xi_{11},\,\xi_{21})$ is arbitrary, we get \eqref{matr_C_posit}.

The minor symmetries are the corollary of the minor symmetries \eqref{C_tens} of the initial tensor $C_{ijkl}$:
$$
(C_1)_{pqs}=C_{pqs1}=C_{qps1}=(C_{1})_{qps},\quad (C^1)_{sqp}=C_{s1qp}=C_{s1pq}=(C^1)_{spq}.
$$
The \eqref{C1_T_permut} is the corollary of the major symmetry of the tensor $C_{ijkl}$:
$$
(C_1(y_1))_{psq}=C_{psq1}(y_1)=C_{q1ps}(y_1)=C_{q1sp}(y_1)=(C^1(y_1))_{qsp}=(C^1(y_1))_{qps}.
$$
\end{proof}

Let the brackets $\langle\cdot\rangle$ denote the mean value of the corresponding function
$$
\langle f(y_1)\rangle=\frac{1}{l_1}\int\limits_{0}^{l_1}f(y_1)dy_1.
$$

\begin{theorem}
\label{thm_hom_tensor}
Let the tensor  $C(y_1)$ satisfy the symmetry and positive conditions \eqref{C_tens}, then the tensor
\begin{equation}
\label{C_hom}
\overline{C}=\langle C(y_1)\rangle-\langle C_1(y_1)\mathcal{C}^{-1}(y_1)C^1(y_1)\rangle+\langle C_1(y_1)\mathcal{C}^{-1}(y_1)\rangle\langle \mathcal{C}^{-1}(y_1)\rangle^{-1} \langle \mathcal{C}^{-1}(y_1) C^1(y_1) \rangle.
\end{equation}
is positive-definite and satisfy the minor and major symmetrical conditions, similar to \eqref{C_tens}.
\end{theorem}

\begin{definition}
We call the tensor \eqref{C_hom} the homogenized tensor.
\end{definition}
This tensor coincides with the same homogenized tensor from \cite{brito-santana_dispersive_2015} (formula (37)). The proof of the theorem \ref{thm_hom_tensor} is given in Appendix \ref{proof_thm_hom_tensor}.

Now we can formulate the homogenized theorem. We start from the homogenized equation.
\begin{definition}
We call the problem 
\begin{gather}
\label{u0_hom_eq}
\nabla_x\cdot  {\bsigma^0}(x,\,t) -\langle \rho(y_1)\rangle\frac{\partial^2}{\partial t^2}{\bf u}^0(x,\,t)=0,\quad
\bsigma^0(x,\,t)=\frac{1}{2} \overline{C}: (\nabla\otimes \bu^0+(\nabla\otimes \bu^0)^T),\\
\nonumber
\bu^0(x,\,0)=\bfv(x),\,\bu^0_t(x,\,0)=\bg(x).
\end{gather}
the homogenized problem. The equation \eqref{u0_hom_eq} we call the homogenized equation.

\end{definition}

Let us define the following tensor
\begin{equation}
\label{tensor_A_big}
\mathcal{A}(y_1)=\mathcal{C}^{-1}(y_1)\Bigl(C^1(y_1)-\langle \mathcal{C}^{-1}(y_1)\rangle^{-1} \langle \mathcal{C}^{-1}(y_1) C^1(y_1)\Bigr).
\end{equation}
We call the tensor $\mathcal{A}(y_1)$ the solution of the cell problem and it has the zero mean $\langle \mathcal{A}(y_1)\rangle=0$.  Note that  $(-\mathcal{A})$ equals to the tensor $N^1$ from \cite{brito-santana_dispersive_2015} (formula (A4)).

\begin{theorem}[{\bf Homogenization theorem}]
\label{thm_hom_eq}
Let function $\bu^0(x,\,t)$ be a solution of the equation \eqref{u0_hom_eq}, then there exists a bounded function $\bu_2(x,\,x_1/\varepsilon,\,t)$ such that the function 
\begin{equation}
\label{u_as_sol_main}
\bu_{as}(x,\,t)=\bu^0(x,\,t)-\varepsilon\left(\int\limits_{0}^{x_1/\varepsilon}\mathcal{A}(\eta)d\eta\right):(\nabla_x \otimes \bu^0(x,\,t))+\varepsilon^2\bu_2(x,\,\frac{x_1}{\varepsilon},\,t),
\end{equation}
satisfies the initial equation (\ref{eq1}) and initial conditions \eqref{init_cond} with the residue $O(\varepsilon)$.

\end{theorem}
The proof of the theorem \ref{thm_hom_eq} is given in the Section \ref{sec_hom_thm_proof}.

Notice, that in \cite{penta_asymptotic_2017} the tensor of the 3-rd order appeared as a solution of the equation for the first correction, but the definition of the homogenized tensor is given as a combination of the tensors of 4-th order, contrary to \eqref{C_hom}.

\subsection{Structure of the solution to the homogenized problem.}
As we mentioned, our method provides the detailed structure of the homogenized problem and its solution. This structure is connected to the certain spectral problem. Let us define the matrix $M(p)$ with elements
\begin{equation}
\label{spec_prob_matrix}
M_{jk}(p)=p_i\overline{C}_{ijkl}p_l,
\end{equation}
where $\overline{C}_{ijkl}$ are the elements of the homogenized tensor \eqref{C_hom}.

\begin{statement}
\label{st_matr_M}
When $p\not=0$,  matrix $M(p)$ is positive definite and symmetric for any $p\in\mathbb{R}^2$.
\end{statement}

\begin{proof}
Let us write down the quadratic form $f^TM(p)f$ for the arbitrary vector $f$. From the Theorem \ref{thm_hom_tensor}, we have
$$
f^TM(p)f=f_jM_{jk}(p)f_k=f_jp_i\overline{C}_{ijkl}p_l f_k\ge A |f_j p_i|^2=A |p|^2 |f|^2.
$$
From this follows the positive definitiveness of the matrix $M(p)$. The symmetry follows from the direct computation of the bilinear from $f^TM(p)v$ and the symmetry of the homogenized tensor $\overline{C}$.
\end{proof}

Let the vectors $\bchi^s_0(p)$ and $(\omega^s_0(p))^2$, $s=1,\,2$, be the solution of the following spectral problem 
\begin{equation}
\label{spec_prob}
\langle\rho(y_1)\rangle(\omega^s_0(p))^2\bchi^s_0(p)=M(p)\bchi^s_0(p),\quad |\bchi_0^s(p)|=1.
\end{equation}

\begin{remark}
The important property of the matrix $M(p)$ is that it is homogeneous of the order $2$
\begin{equation}
\label{matr_hom}
M(\lambda p)=\lambda^2 M(p).
\end{equation}
Thus we can define the unitary vector $\bn(\phi)=(\cos\phi,\,\sin\phi)$, where $\phi\in[0,\,2\pi]$ and represent the solution of the spectral problem \eqref{spec_prob} in the polar coordinates. Let $p=r\bn(\phi)$, then we have the following equalities $\omega^s(r\bn(\phi))=r \lambda^s(\phi)$ and $\bchi^s_0(r\bn(\phi))=\bchi^s_0(\phi)$, where we have the reduced spectral problem
\begin{equation}
\label{spec_prob_angle}
\frac{1}{\langle \rho(y_1)\rangle}M(\bn(\phi))\bchi^s_0(\phi)=(\lambda^s(\phi))^2\bchi^s_0(\phi),\quad \phi\in[0,\,2\pi],\,
|\bchi_0^s(\phi)|=1,\,s=1,\,2.
\end{equation}

\end{remark}

The vector-functions $\bchi_0^s(p)$ define the pseudo-differential operators $\bchi_0^s(-i\varepsilon\nabla_x)$ and the eigenvalues $\omega_0^s(p)$ define two scalar pseudo-differential equations with constant coefficients
\begin{gather}
\label{v_eq}
-v^s_{tt}(x,\,t)=(\omega^s_0(-i\nabla_x))^2v^s(x,\,t),\\
\nonumber
 v^s|_{t=0}=\langle \bchi_0^s(-i\nabla_x),\,\bfv(x)\rangle,\,
\frac{d}{dt}v^s|_{t=0}=\langle \bchi_0^s(-i\nabla_x),\,\bg(x)\rangle.
\end{gather}

\begin{definition}
We call the function 
\begin{equation}
\label{u0_princ_hom}
\bu^0(x,\,\,t)=\sum\limits_{s}\bchi_0^s(-i\nabla_x)v^s(x,\,t),
\end{equation}
the principal term of the asymptotic expansion. 
\end{definition}
The function \eqref{u0_princ_hom} describes the main effects of the homogenized wave and it is natural to call it the result of the homogenization procedure.

\begin{statement}
\label{st_princ_term_eq}
Function $\bu^0(x,\,t)$ defined in (\ref{u0_princ_hom}) satisfies the homogenized problem \eqref{u0_hom_eq}.
\end{statement}

\begin{proof}
The proof of this theorem is a straightforward calculations. We start from the time derivative, we have the following equality
\begin{gather}
\nonumber
-\langle \rho(y_1)\rangle\frac{\partial^2}{\partial t^2}\bchi_0^s(-i\nabla_x)v^s(x,\,t)=\langle \rho(y_1)\rangle\bchi_0^s(-i\nabla_x)\left(-\frac{\partial^2}{\partial t^2} v^s(x,\,t)\right)=\\
\label{left_side_eq}
=\langle \rho(y_1)\rangle\bchi_0^s(-i\nabla_x)\left(\left(\omega_0^s(-i\nabla_x)\right)^2v^s(x,\,t)\right)
=\langle \rho(y_1)\rangle\left(\omega_0^s(-i\nabla_x)\right)^2\bchi_0^s(-i\nabla_x)v^s(x,\,t).
\end{gather}
Now let us calculate the action of the spatial operator. First of all, we have the following equality
\begin{equation}
\label{stress_tensor}
\bsigmaE(x,\,t)=C(\frac{x_1}{\varepsilon}):(\nabla_x\otimes \bu^\varepsilon(x,\,t)).
\end{equation}
This follows from the symmetry property  (\ref{C_tens}) of tensor: we can rewrite $\bsigmaE$ as follows
$$
\bsigmaE(x,\,t)=\frac{1}{2}C_{ijkl}(\frac{x_1}{\varepsilon})\Bigl(\frac{\partial u_k^\varepsilon}{\partial x_l}+\frac{\partial u_l^\varepsilon}{\partial x_k}\Bigr)= 
C_{ijkl}(\frac{x_1}{\varepsilon})\frac{\partial u_k^\epsilon}{\partial x_l}.
$$
Here we used the conventional summation over repeating indices. 

After that we have to calculate the action, using the theory of the pseudo-differential operators. Following Appendix \ref{PsiDO}, we use the Fourier transform
\begin{equation}
\label{Fourier_trans}
\hat{v}(p,\,t)\equiv F[v](p,\,t)=\frac{1}{2\pi}\int\limits_{\mathbb{R}^2}v(x,\,t)e^{-i p\cdot x}dx,\quad
v(x,\,t)\equiv F^{-1}[\hat{v}](x,\,t)=\frac{1}{2\pi}\int\limits_{\mathbb{R}^2}\hat{v}(p,\,t)e^{i p\cdot x}dp.
\end{equation}
The action of the (pseudo-)differential operator can be defined as follows
\begin{gather*}
(-i\nabla_x)\cdot \overline{C} :((-i\nabla_x)\otimes\bchi^s_0(-i\nabla_x))v^s(x,\,t)=F^{-1}[\bp \cdot \overline{C} :(\bp \otimes\bchi^s_0(\bp))\hat{v}^s(\bp,\,t)].
\end{gather*}
Now we use the the definition of the matrix \eqref{spec_prob_matrix} and fact that $((\omega_0^s(p))^2,\,\bchi_0^s(p))$  is the solution of the spectral problem (\ref{spec_prob}). We have
\begin{gather*}
F^{-1}[\bp \cdot \overline{C} :(\bp \otimes\bchi^s_0(\bp))\hat{v}^s(\bp,\,t)]=F^{-1}[M(\bp)\bchi^s_0(\bp)\hat{v}^s(\bp,\,t)]=F^{-1}[\langle \rho(y_1)\rangle(\omega_0^s(\bp))^2\bchi_0(\bp)\hat{v}^s(\bp,\,t)\rangle]=\\
=\langle \rho(y_1)\rangle(\omega_0^s(-i\nabla_x))^2\bchi_0(-i\nabla_x)v^s(x,\,t).
\end{gather*}
We see that for each  $s$, function $\bchi_0^s(-i\nabla_x)v^s(x,\,t)$ satisfies the equation \eqref{u0_hom_eq}. Since the equation is linear, the sum  (\ref{u0_princ_hom}) also satisfies the equation.	
\end{proof}

This leads to the different way of determination of the function $\bu^0(x,\,t)$: instead of the direct computation of the solution of the homogenized problem \eqref{u0_hom_eq}, one can solve the equations \eqref{v_eq} and then write down the sum \eqref{u0_princ_hom}. This leads to the Fourier transform representation of the function $\bu^0(x,\,t)$.

\begin{statement}
\label{st_u0_Fourier}
Let $\hat{\bfv}(p)$ and $\hat{\bg}(p)$ be the Fourier transform of the initial conditions \eqref{init_cond} and $\bn(\phi)=(\cos\phi,\,\sin\phi)$, then the solution of the problem \eqref{u0_hom_eq} has the form 
\begin{gather}
\label{u0_Fourier_polar}
\bu^0(x,\,t)=\frac{1}{2\pi}\int\limits_{0}^{2\pi}\int\limits_{0}^{+\infty}\left[\sum\limits_{s=1,\,2}\left(\langle\hat{\bfv}(r\bn(\phi)),\,\bchi_0^s(\phi)\rangle\cos(r t \lambda^s(\phi))+\right.\right.\\
\nonumber
\left.\left.+i\frac{\langle\hat{\bg}(r\bn(\phi)),\,\bchi_0^s(\phi)\rangle}{r \lambda^s(\phi)}\sin(rt\lambda^s(\phi))\right)\bchi_0^s(\phi)e^{i r \langle\bn(\phi), x\rangle}\right]rdrd\phi,
\end{gather}
where $(\lambda^s(\phi),\,\bchi_0^s(\phi))$ is the solution of the reduced spectral problem \eqref{spec_prob_angle}.

\end{statement}

\begin{proof} Following \eqref{u0_princ_hom} we need to calculate the action of the pseudo-differential operator $\bchi_0^s(-i\nabla_x)$ on the function $v^s(x,\,t)$. This can be done via the Fourier transform \eqref{Fourier_trans}
$$
\bchi_0^s(-i\nabla_x)v^s(x,\,t)=F^{-1}[\bchi^s_0(p)\hat{v}^s(p,\,t)](x)=\frac{1}{2\pi}\int\limits_{\mathbb{R}^2}\bchi^s_0(p)\hat{v}^s(p,\,t)e^{i p\cdot x}dp.
$$
Here we denoted $\hat{v}^s(p,\,t)$ the Fourier transform of the function $v^s(x,\,t)$.

The function $v^s(x,\,t)$ is the solution of the problem \eqref{v_eq}, which is the equation with constant coefficients. This leads to analytical representation of the Fourier transform
$$
\hat{v}^s(p,\,t)=\langle\hat{\bfv}(p),\,\bchi_0^s(p)\rangle\cos(t\omega_0^s(p))+i\frac{\langle\hat{\bg}(p),\,\bchi_0^s(p)\rangle}{\omega_0^s(p)}\sin(t\omega_0^s(p)).
$$
We substitute this representation into the Fourier integral and then pass to the polar coordinates $p=r\bn(\phi)$ and use \eqref{spec_prob_angle}.

\end{proof}

Function \eqref{u0_Fourier_polar} defines the general form of the solution. One can simplify this representation and calculate analytically the integral over $r\in[0,\,+\infty)$ if some additional assumptions about the form of the initial functions $\bfv(x)$ and $\bg(x)$ are given. Similar to ideas of \cite{dobr_sekerzh_volk_2009,  grushin_homogenization_2013, dobrokhotov_asymptotic_2016}, we use the following form of the initial conditions
\begin{equation}
\label{ist_sekr}
\bfv(x)=V(x) \bfv,\quad \bg(x)=V(x) \bg,\quad V(x)=\frac{A}{(1+|x|^2)^{3/2}},\quad \bfv,\,\bg\in\mathbb{R}^2.
\end{equation}
Here vectors $\bfv$ and $\bg$ are the constant vectors and the perturbation is described by the function $V(x)$. This function has very simple Fourier transform
\begin{equation}
\label{ist_seker_Fourier}
\hat{V}(p)=A e^{-|p|}.
\end{equation}

\begin{statement}
\label{st_u0_ist_seker}
Let initial functions are given by \eqref{ist_sekr}, then the solution of the problem \eqref{u0_hom_eq} has the form 
\begin{gather}
\label{u0_ist_seker}
\bu^0(x,\,t)=\frac{A}{2\pi}\int\limits_{0}^{2\pi}
\left[\sum\limits_{s=1,\,2}\left(-\langle \bfv,\,\bchi_0^s(\phi)\rangle \frac{(t\lambda^s(\phi))^2 +(i+\langle \bn(\phi),\,x\rangle)^2}{((t\lambda^s(\phi))^2-(i+\langle \bn(\phi),\,x\rangle)^2)^2}+\right.\right.\\
\nonumber
\left.\left.+i\frac{t \langle \bg,\,\bchi_0^s(\phi)}{(t\lambda^s(\phi))^2-(i+\langle \bn(\phi),\,x\rangle)^2}\right)\bchi_0^s(\phi)\right]d\phi,
\end{gather}
where $(\lambda^s(\phi),\,\bchi_0^s(\phi))$ is the solution of the reduced spectral problem \eqref{spec_prob_angle}.

\end{statement}

\begin{proof} The proof follows from the direct computation of the integrals over $r\in[0,\,+\infty)$:
$$
\int\limits_{0}^{+\infty}r \cos(r\alpha)A e^{-r} e^{i r\beta}dr=-A \frac{\alpha^2 +(i+\beta)^2}{(\alpha^2-(i+\beta)^2)^2},\quad
\int\limits_{0}^{+\infty} \sin(r\alpha)A e^{-r} e^{i r\beta}dr=A \frac{\alpha}{\alpha^2-(i+\beta)^2}.
$$
Now we identify $\alpha=t\lambda^s(\phi)$ and $\beta=\langle \bn(\phi),\,x\rangle$.
\end{proof}

In this case we can also present the simple case for the function $v^s(x,\,t)$.

\begin{corollary}
\label{cor_v_func_ist_seker}
Let initial functions are given by \eqref{ist_sekr}, then the solution of the problem \eqref{v_eq} is of the form
\begin{equation}
\label{v_ist_seker}
v^s(x,\,t)=\frac{A}{2\pi}\int\limits_{0}^{2\pi}
\left(-\langle \bfv,\,\bchi_0^s(\phi)\rangle \frac{(t\lambda^s(\phi))^2 +(i+\langle \bn(\phi),\,x\rangle)^2}{((t\lambda^s(\phi))^2-(i+\langle \bn(\phi),\,x\rangle)^2)^2}+i\frac{t \langle \bg,\,\bchi_0^s(\phi)}{(t\lambda^s(\phi))^2-(i+\langle \bn(\phi),\,x\rangle)^2}\right)d\phi.
\end{equation}
\end{corollary}

The approach, based on the Statements \ref{st_u0_Fourier} and \ref{st_u0_ist_seker} can be  implemented relatively simply. The matrix \eqref{spec_prob_matrix} is the matrix of the dimensions $2\times 2$ and the spectral problem \eqref{spec_prob} can be easily resolved. After that we can use the function \eqref{ist_sekr} and this leads to the calculation of the ordinary integral.

\section{Analysis of some numerical examples}
\label{sec_num_examp}

Here we present the numerical examples for the calculation of the effective homogenized tensor. We start from the example of the orthorombic material and then provide some numerical calculations for the monoclinic perturbation given by the rotation.

Following \cite{royer_elastic_2000} and Voigt notation, we introduce the index $\alpha$ in a following way
$$
(ij),\,(kl)\to \alpha:\,\quad (11)\to 1,\,(22)\to 2,\, (12)=(21)\to 6.
$$
We do not use other values of the indices $i,j,k,l$ due to our model.  Under this substitution of the pairs of indices we have the following
\begin{equation}
\label{index_subst}
C_{1111}=c_{11},\,C_{2222}=c_{22},\,C_{1122}=c_{12},\, C_{2121}=c_{66},\,C_{1121}=c_{16},\, C_{2122}=c_{26}.
\end{equation}
We do not mention other permutations of indices due to the symmetry.

We consider the orthorombic situation, see for example \cite{royer_elastic_2000} (formula (3.64)), and we need the following elements of the tensor 
\begin{equation}
\label{tensors_orthotrop}
c_{11}\not=0,\,c_{22}\not=0,\,c_{12}\not=0,\,c_{16}=0,\,c_{26}=0,\,c_{66}\not=0.
\end{equation}

To demonstrate our results for the orthorombic case we chose the materials presented in the table \ref{tbl1_base elements}, following \cite{bledzki_determination_1999, TMM}. 
\begin{table}[!h]
\begin{center}
\begin{tabular}{|c|c|c|c|c|c|}
\hline 
 & $E_1$ GPa & $E_2$ GPa & $G_{12}$ GPa & $\nu_{12}$ & $\rho$ kg/m$^3$\\ 
\hline 
Material A: glass/epoxy: & 39.15 & 12.16 & 5.11 & 0.296 & 1892 \\ 
\hline 
Material B:  kevlar-49 /epoxy:  & 73.0 & 5.0 & 2.2 & 0.35 & 1400 \\ 
\hline 
\end{tabular} 
\end{center}
\caption{Table with data of the materials for the composite material \cite{bledzki_determination_1999, TMM}. \label{tbl1_base elements}}
\end{table}

We use the following formulas to determine the elastic constants $c_{ij}$
\begin{equation}
\label{elstic_const_calc}
c_{11}=\frac{E_1}{\Delta}, \quad  c_{22}=\frac{E_2}{\Delta}, \quad c_{12}=\frac{\nu_{12}E_2}{\Delta}, \quad c_{66}=G_{12},\quad \nu_{21}=\nu_{12}\frac{E_2}{E_1},\quad \Delta=1-\nu_{12}\nu_{21}.
\end{equation}
The result of the calculation is given in the table  \ref{tbl2_elstic_const}.

\begin{table}[!h]
\begin{center}
\begin{tabular}{|c|c|c|c|c|}
\hline 
 & $c_{11}$ GPa & $c_{12}$ GPa & $c_{22}$ GPa & $c_{66}$ GPa \\ 
\hline 
Material A: glass/epoxy: & 40.2452 & 3.70005 & 12.5002 & 5.11 \\ 
\hline 
Material B:  kevlar-49 /epoxy: & 73.6177 & 1.76481 & 5.04231 & 2.2 \\ 
\hline 
\end{tabular} 
\end{center}
\caption{The elastic constants of the materials. \label{tbl2_elstic_const}}
\end{table}

The homogenized tensor  \eqref{C_hom} has the following non-zero components (taking into account the symmetries).
\begin{gather}
\label{C_hom_orthoromb}
\overline{C}_{1111}=\langle c^{-1}_{11}\rangle^{-1},\quad \overline{C}_{1122}=\langle\frac{c_{12}}{c_{11}}\rangle \langle c^{-1}_{11}\rangle^{-1},\quad \overline{C}_{1212}=\langle c^{-1}_{66}\rangle^{-1},\\
\nonumber
\overline{C}_{2222}=\langle c_{22}\rangle-\langle \frac{c_{12}^2}{c_{11}}\rangle+\langle \frac{c_{12}}{c_{11}}\rangle^2 \langle c^{-1}_{11}\rangle^{-1}.
\end{gather}

After evaluation of the effective homogenized tensor \eqref{C_hom},  we can construct, using \eqref{spec_prob_matrix},  the effective matrix $M(p)$ 
$$
M(p)=\begin{pmatrix}
p_1^2 \overline{C}_{1111}+p_2^2 \overline{C}_{1212} & p_1 p_2 \left(\overline{C}_{1122}+\overline{C}_{1212}\right)\\
p_1 p_2 \left(\overline{C}_{1122}+\overline{C}_{1212}\right) & p_1^2 \overline{C}_{1212} + p_2^2\overline{C}_{2222}
\end{pmatrix}.
$$

We put the length of the cell equals $\ell=1$ and we assume that the distribution of materials is following: for $y\in[0,\,\gamma \ell]$ we have material A and for $y\in(\gamma \ell,\,\ell]$ we have material B. In the table \ref{tbl3_hom_tensor_gamma} we provide the effective constants after homogenization procedure, corresponding to \eqref{C_hom_orthoromb}. 
\begin{table}[!h]
\begin{center}
\begin{tabular}{|c|c|c|c|c|c|}
\hline 
$\gamma$  & $c_{11}$ & $c_{12}$ & $c_{22}$ & $c_{66}$ & $\rho$ \\ 
\hline 
0.2  & 63.1453 & 2.37209 & 6.52111 & 2.48277 & 1498.4 \\ 
\hline 
0.4  & 55.2813 & 2.82812 & 8.00868 & 2.84896 & 1596.8 \\ 
\hline 
0.6  & 49.1592 & 3.18314 & 9.50211 & 3.34185 & 1695.2 \\ 
\hline 
0.8  & 44.2578 & 3.46736 & 10.9996 & 4.04098 & 1793.6 \\ 
\hline 
\end{tabular} 
\end{center}
\caption{Elastic constants  of the homogenized tensor and the density,  depending on the fraction of the volume $\gamma$. \label{tbl3_hom_tensor_gamma}}
\end{table}

On the figure \ref{pic_orthorombic_eigenvalues} we demonstrated the solution of the generalized spectral problem \eqref{spec_prob} for the value of the volume fraction $\gamma=0.2$ in the orthorombic case. We also demonstrated the eigenvalues $\lambda^2(\phi)$ of the reduced spectral problem \eqref{spec_prob_angle}.
\begin{figure}[!h]
\begin{center}
\begin{subfigure}{0.3\textwidth}
\includegraphics[scale=0.35]{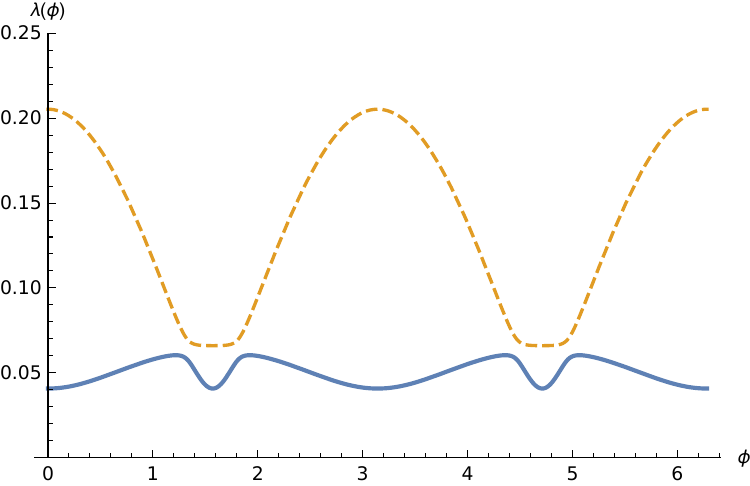}
\caption{\label{pic_orthorombic_eigenvalues:a}}
\end{subfigure}
\begin{subfigure}{0.3\textwidth}
\includegraphics[scale=0.35]{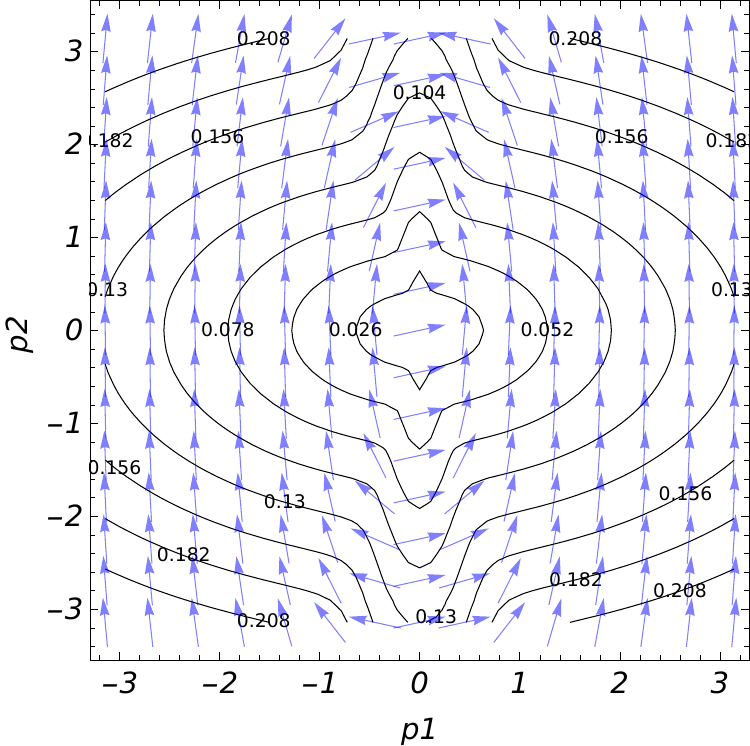}
\caption{\label{pic_orthorombic_eigenvalues:b}}
\end{subfigure}
\begin{subfigure}{0.3\textwidth}
\includegraphics[scale=0.35]{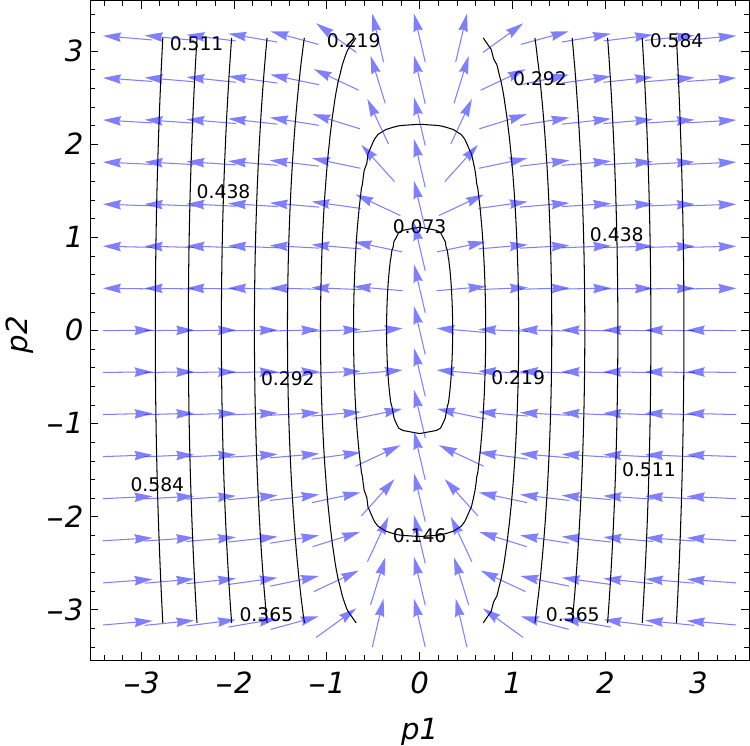}
\caption{\label{pic_orthorombic_eigenvalues:c}}
\end{subfigure}
\end{center}
\caption{Figure \ref{pic_orthorombic_eigenvalues:a} --- the eigenvalues $\lambda^{1,2}(\phi)$ of the reduced spectral problem \eqref{spec_prob_angle}. Figures \ref{pic_orthorombic_eigenvalues:b} and  \ref{pic_orthorombic_eigenvalues:c} --- the eigenvectors $\bchi_0(p)$ of the spectral problem \eqref{spec_prob} for the orthorombic case. Contour lines are the level lines of the eigenvalues $\omega_0(p)$ and the arrows show the vector-field $\bchi_0(p)$.  Figure  \ref{pic_orthorombic_eigenvalues:b} --- the first pair of the eigenvector and eigenvalue. Figure   \ref{pic_orthorombic_eigenvalues:c} --- the second pair. The value $\gamma=0.2$. \label{pic_orthorombic_eigenvalues} }
\end{figure}

After we calculated the solution of the spectral problem \eqref{spec_prob},  we can calculate the solution of the homogenized Cauchy problem \eqref{u0_hom_eq}. We chose the initial conditions as follows: initial perturbation is of the form $\bfv(x)=V(x)(1,\,0)^T$, where function $V(x)$ is given in \eqref{ist_sekr} and the initial velocity $\bg(x)=0$. We demonstrate the functions $v^s(x,\,t)$ defined by \eqref{v_ist_seker} on the figures \ref{pic_v1_gamma_02} and \ref{pic_v2_gamma_02}. We demonstrate the wave profiles for the value $\gamma=0.2$. We chose the time moment $t=60$. On the figure \ref{pic_wave_U1_gamma_02} we demonstrated the first component of the solution of the homogenized problem \eqref{u0_hom_eq}. On the figure \ref{pic_wave_U2_gamma_02} we demonstrated the second part of the solution. The solution of the problem \eqref{u0_hom_eq} was obtained by the direct numerical solution.
\begin{figure}[!h]
\begin{center}
\includegraphics[scale=0.5]{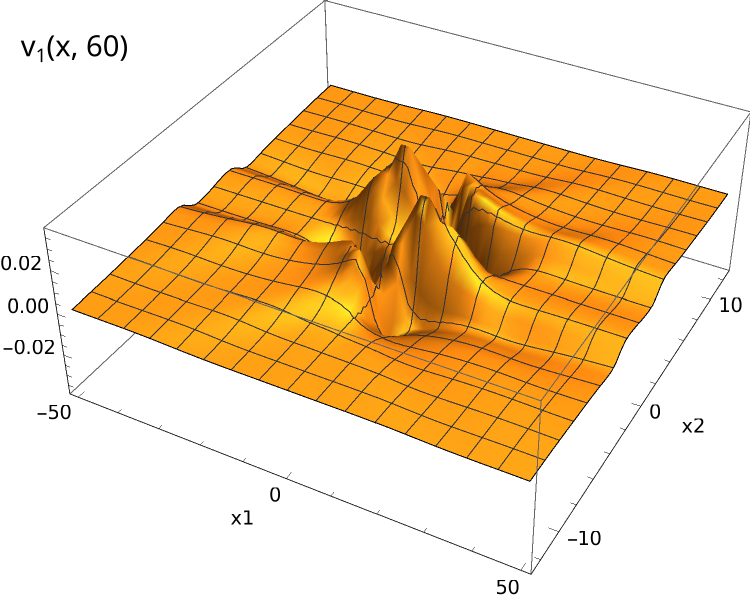}
\includegraphics[scale=0.65]{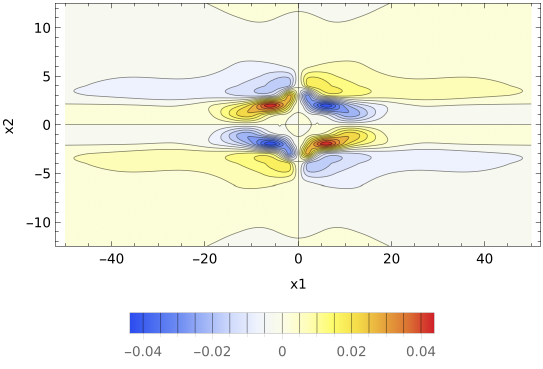}
\end{center}
\caption{The function $v^1(x,\,t)$, defined by \eqref{v_ist_seker}, corresponding to the first solution of the spectral problem \eqref{spec_prob}.  On the left we have the 3D plot of the function, on the right we have the level lines of the same function. The value $\gamma=0.2$, the time is $t=60$. \label{pic_v1_gamma_02} }
\end{figure}
\begin{figure}[!h]
\begin{center}
\includegraphics[scale=0.5]{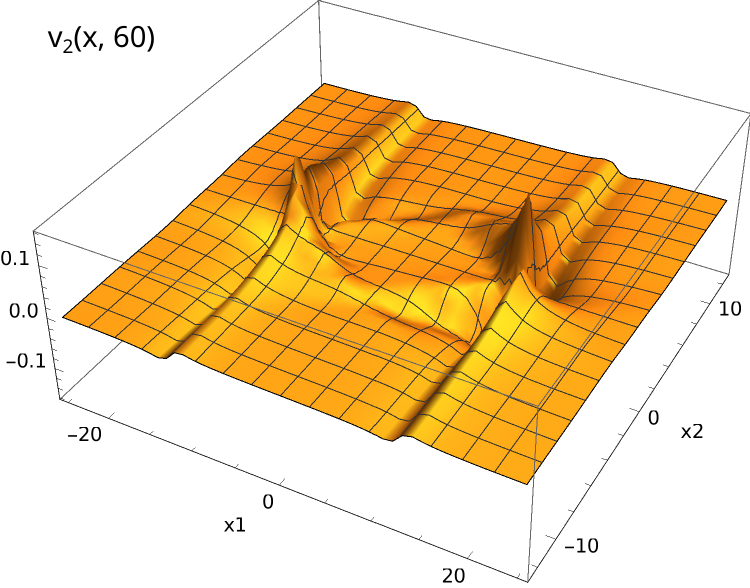}
\includegraphics[scale=0.65]{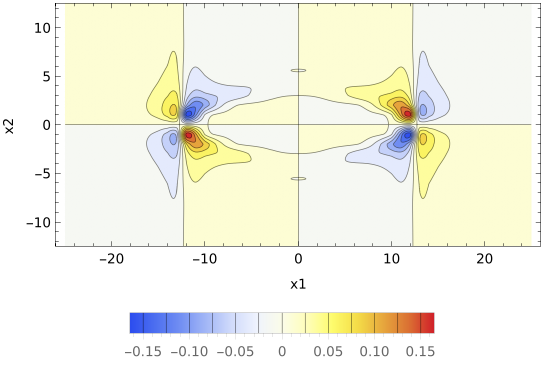}
\end{center}
\caption{The function $v^2(x,\,t)$, defined by \eqref{v_ist_seker}, corresponding to the second solution of the spectral problem \eqref{spec_prob}.  On the left we have the 3D plot of the function, on the right we have the level lines of the same function. The value $\gamma=0.2$, the time is $t=60$. \label{pic_v2_gamma_02} }
\end{figure}

\begin{figure}[!h]
\begin{center}
\includegraphics[scale=0.65]{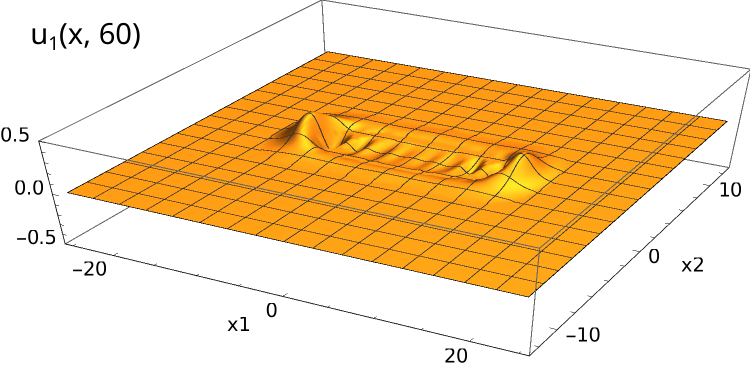}
\includegraphics[scale=0.65]{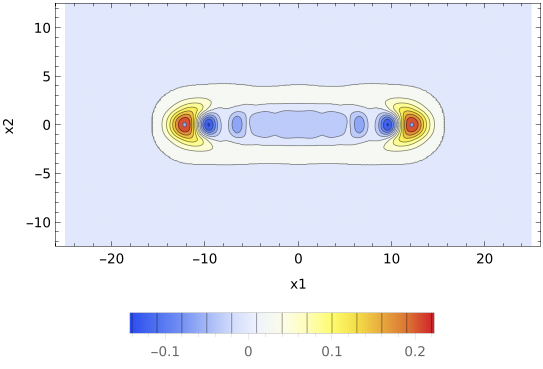}
\end{center}
\caption{The first component of the solution of the homogenized problem \eqref{u0_hom_eq}. On the left we have the 3D plot of the wave, on the right we have the level lines of the same wave. The value $\gamma=0.2$, the time is $t=60$. \label{pic_wave_U1_gamma_02} }
\end{figure}
\begin{figure}[!h]
\begin{center}
\includegraphics[scale=0.65]{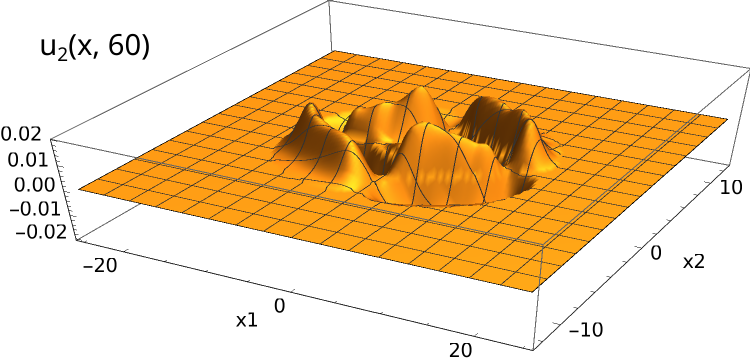}
\includegraphics[scale=0.65]{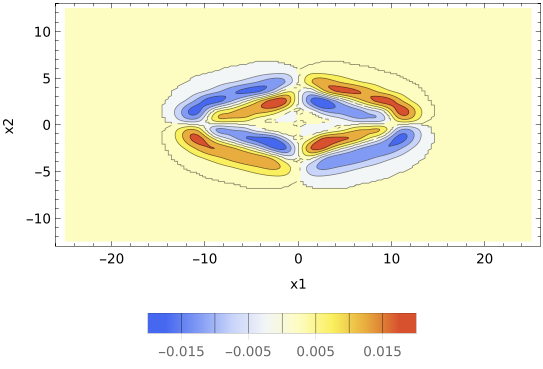}
\end{center}
\caption{The second component of the solution of the homogenized problem \eqref{u0_hom_eq}. On the left we have the 3D plot of the wave, on the right we have the level lines of the same wave. The value $\gamma=0.2$, the time is $t=60$. \label{pic_wave_U2_gamma_02} }
\end{figure}

\subsection{Monoclinic case. Rotation.}
The monoclinic case differs from the orthorombic case by $c_{16}\not=0$ and $c_{26}\not=0$. In this case we have the following formulas for the homogenized tensor \eqref{C_hom}.
Let us define the following values 
\begin{gather*}
\alpha=\left\langle \frac{c_{11}(y_1)}{d(y_1)}\right\rangle,\, \beta=\left\langle \frac{c_{66}(y_1)}{d(y_1)}\right\rangle,\, \gamma=-\left\langle \frac{c_{16}(y_1)}{d(y_1)}\right\rangle,\quad 
d(y_1)=c_{11}(y_1)c_{66}(y_1)-c_{16}^2(y_1).
\end{gather*}
Let us denote 
$$
P=\left\langle \frac{c_{12}(y_1)c_{16}(y_1)-c_{11}(y_1)c_{26}(y_1)}{d(y_1)}\right\rangle,\,
Q=\left\langle \frac{c_{16}(y_1)c_{26}(y_1)-c_{12}(y_1)c_{66}(y_1)}{d(y_1)}\right\rangle.
$$
Let $\Delta=\alpha\beta-\gamma^2$, then the components of the homogenized tensor are calculated by the following formulas
\begin{gather}
\label{hom_tens_monoclinic}
\overline{C}_{1111}=\frac{\alpha}{\Delta},\,\overline{C}_{1121}=-\frac{\gamma}{\Delta},\,\overline{C}_{2121}=\frac{\beta}{\Delta},\, \overline{C}_{1122}=\frac{\gamma P-\alpha Q}{\Delta},\,\overline{C}_{2221}=\frac{\gamma Q-\beta P}{\Delta},\\
\nonumber
\overline{C}_{2222}=\langle c_{22}(y_1)\rangle+\frac{\alpha Q^2-2\gamma P Q+\beta P^2}{\Delta}-\left\langle \frac{c_{26}^2c_{11}+c_{12}^2c_{66}-2c_{26}c_{12}c_{16}}{d}\right\rangle.
\end{gather}

In the case $c_{16}=c_{26}=0$ we have $\alpha=\langle c_{66}^{-1}\rangle$, $\beta=\langle c_{11}^{-1}\rangle$, $\gamma=0$, $\Delta=\alpha\beta$, $P=0$ and $Q=-\langle c_{12}/c_{11}\rangle$.
We see, that in this case, the formulas \eqref{hom_tens_monoclinic} coincide with \eqref{C_hom_orthoromb}.

As an example of the monoclinic tensor we consider the rotation, see for example \cite{tsai_theory_composites_design_2012} (formula (2.35)). Let us define the matrix
$$
T(\theta)=
\begin{pmatrix}
c^2 & s^2 & c s\\
s^2 & c^2 & - c s\\
-2cs & 2 cs & c^2-s^2
\end{pmatrix},\quad
c=\cos\theta,\,s=\sin\theta.
$$
Let us consider the matrix of the coefficients, corresponding to the orthorombic case
$$
c(y_1)=\begin{pmatrix}
c_{11}(y_1) & c_{12}(y_1) & 0\\
c_{12}(y_1) & c_{22}(y_1) & 0\\
0 & 0 & c_{66}(y_1)
\end{pmatrix},
$$
then application of the rotation matrix leads to the following monoclinic matrix of the coefficients
\begin{equation}
\label{monoclinic_rotation}
c_M(y_1)=T^{-T}(\theta)c(y_1) T^{-1}(\theta)
\end{equation}
Here $-T$ stands for the inverse and transpose matrix. 
The homogenization procedure is applied to the tensor $c_M(y_1)$.

As an example we provide the rotation of the orthorombic matrix, constructed above, by the angle $\theta=\pi/12$. We have the following values for the homogenized tensor $\overline{C}$ for the matrix $c_M(y_1)$ according to the formulas \eqref{hom_tens_monoclinic}.

\begin{table}[!h]
\begin{center}
\begin{tabular}{|c|c|c|c|c|c|c|}
\hline 
$\gamma$  & $c_{11}$ & $c_{12}$ & $c_{22}$ & $c_{66}$ & $c_{16}$ & $c_{26}$ \\ 
\hline 
0.2  & 55.7328 & 5.64414 & 6.94335 & 5.9176 & -13.109 & -1.27751  \\ 
\hline 
0.4  & 48.9719 & 5.45971 & 8.40035 & 5.72593 & -11.0417 & -1.17266 \\ 
\hline 
0.6  & 43.8232 & 5.33626 & 9.86142 & 5.75051 & -9.31275 & -1.06463 \\ 
\hline 
0.8  & 39.8507 & 5.26648 & 11.3277 & 6.02477 & -7.74734 & -0.939101 \\ 
\hline 
\end{tabular} 
\end{center}
\caption{The result of the homogenization procedure of the monoclinic tensor $c_M(y_1)$ after rotation procedure, depending on the different volume fraction $\gamma$. \label{tbl_monoclinic}}
\end{table}

On the picture \ref{pic_monoclinic_eigenvalues} we demonstrated the solution of the generalized spectral problem \eqref{spec_prob} for the value of the volume fraction $\gamma=0.2$ in the monoclinic case. We also demonstrated the eigenvalues $\lambda^2(\phi)$ of the reduced spectral problem \eqref{spec_prob_angle}.
\begin{figure}[!h]
\begin{center}
\begin{subfigure}{0.3\textwidth}
\includegraphics[scale=0.35]{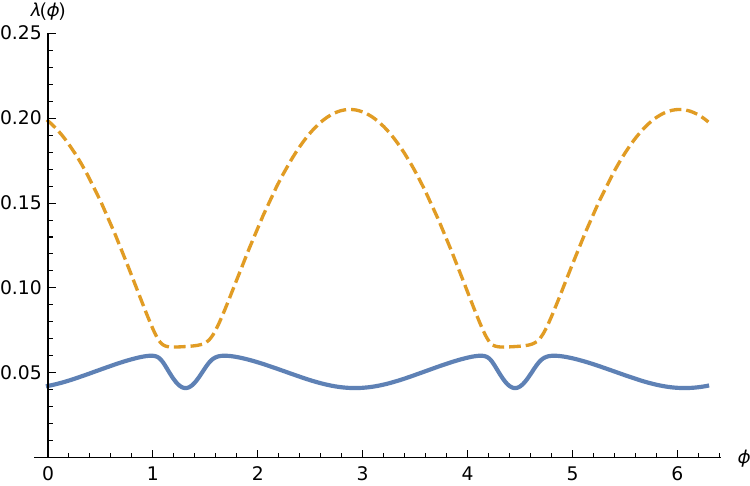}
\caption{\label{pic_monoclinic_eigenvalues:a}}
\end{subfigure}
\begin{subfigure}{0.3\textwidth}
\includegraphics[scale=0.35]{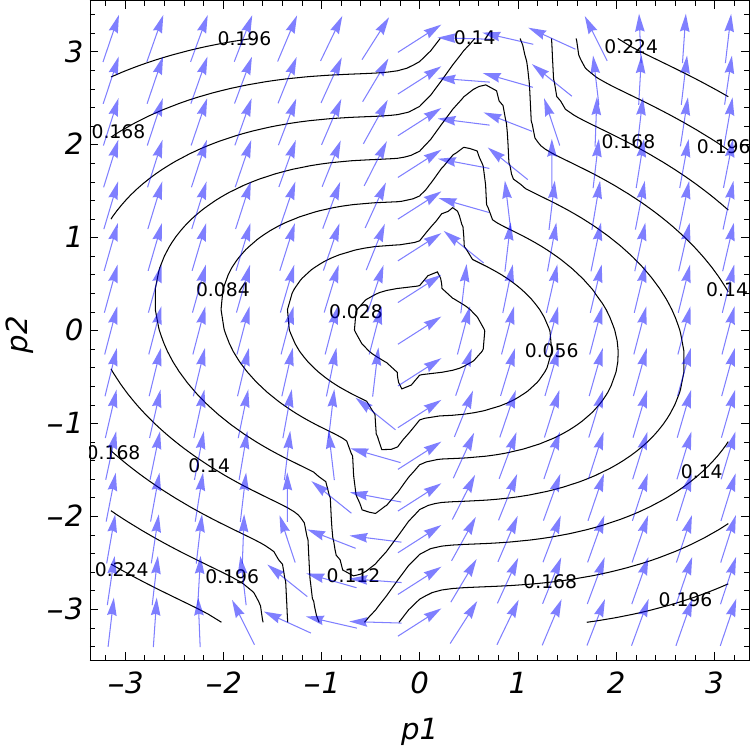}
\caption{\label{pic_monoclinic_eigenvalues:b}}
\end{subfigure}
\begin{subfigure}{0.3\textwidth}
\includegraphics[scale=0.35]{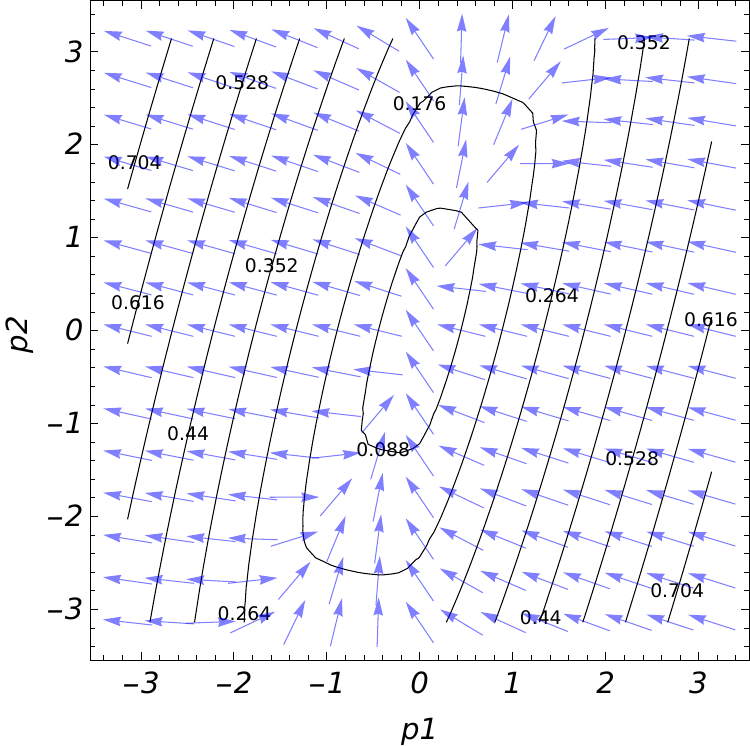}
\caption{\label{pic_monoclinic_eigenvalues:c}}
\end{subfigure}
\end{center}
\caption{The monoclinic case. Figure \ref{pic_monoclinic_eigenvalues:a} --- the eigenvalues $\lambda^{1,2}(\phi)$ of the reduced spectral problem \eqref{spec_prob_angle}. Figures \ref{pic_monoclinic_eigenvalues:b} and  \ref{pic_monoclinic_eigenvalues:c} --- the eigenvectors $\bchi_0(p)$ of the spectral problem \eqref{spec_prob} for the orthorombic case. Contour lines are the level lines of the eigenvalues $\omega_0(p)$ and the arrows show the vector-field $\bchi_0(p)$.  Figure  \ref{pic_monoclinic_eigenvalues:b} --- the first pair of the eigenvector and eigenvalue. Figure   \ref{pic_monoclinic_eigenvalues:c} --- the second pair. The value $\gamma=0.2$. \label{pic_monoclinic_eigenvalues} }
\end{figure}

Here we are presenting the results of the computation of the homogenized problem \eqref{u0_hom_eq} for the monoclinic case.
The parameters are the same: initial perturbation is of the form $\bfv(x)=V(x)(1,\,0)^T$ (function $V(x)$ is defined in \eqref{ist_sekr}) and the initial velocity $\bg(x)=0$, the value $\gamma=0.2$ and the time moment $t=60$. We demonstrate the functions $v^s(x,\,t)$ defined by \eqref{v_ist_seker} on the figures \ref{pic_monoclinic_v1_gamma_02} and \ref{pic_monoclinic_v2_gamma_02}. On the picture \ref{pic_wave_monoclinic_U1_gamma_02} we demonstrated the first component of the solution of the homogenized problem \eqref{u0_hom_eq}. On the picture \ref{pic_wave_monoclinic_U2_gamma_02} we demonstrated the second part of the solution. The solution of the problem obtained by direct numerical solution of the equation.
\begin{figure}[!h]
\begin{center}
\includegraphics[scale=0.5]{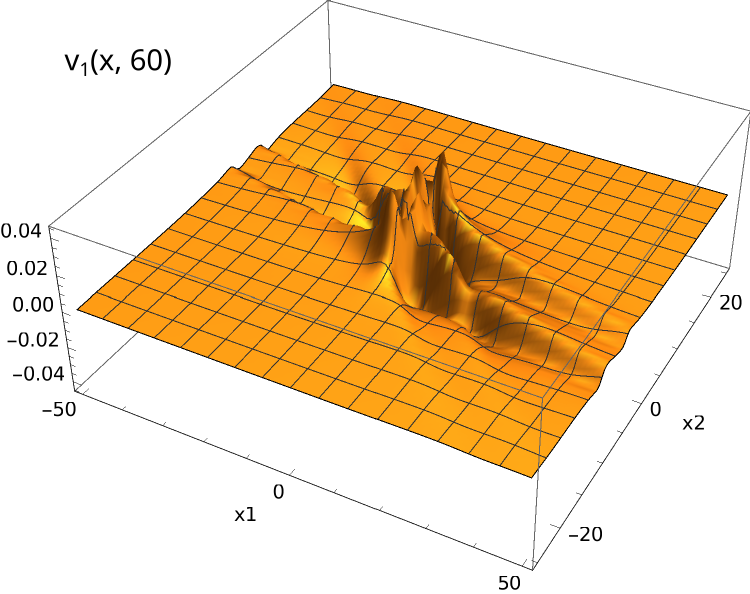}
\includegraphics[scale=0.65]{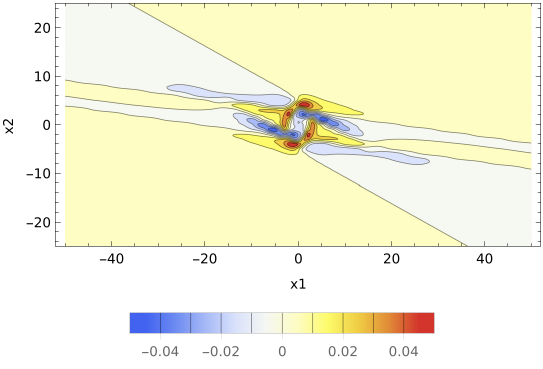}
\end{center}
\caption{The monoclinic case. The function $v^1(x,\,t)$, defined by \eqref{v_ist_seker}, corresponding to the first solution of the spectral problem \eqref{spec_prob}.  On the left we have the 3D plot of the function, on the right we have the level lines of the same function. The value $\gamma=0.2$, the time is $t=60$. \label{pic_monoclinic_v1_gamma_02} }
\end{figure}
\begin{figure}[!h]
\begin{center}
\includegraphics[scale=0.5]{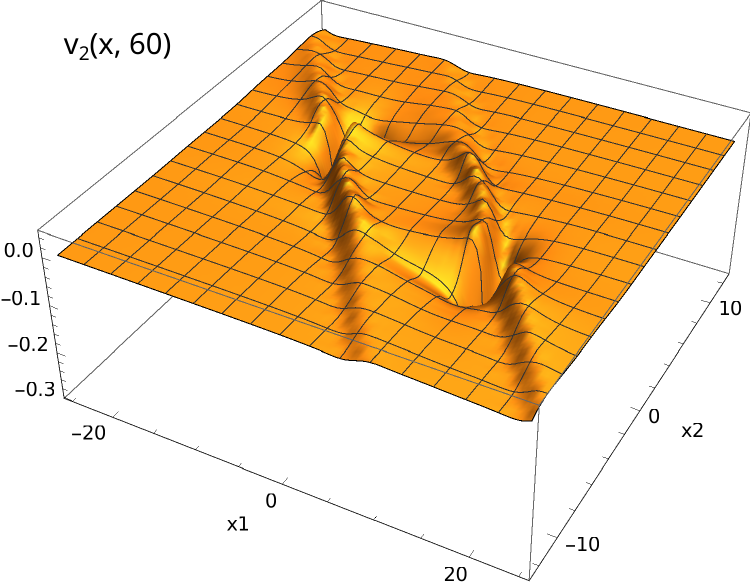}
\includegraphics[scale=0.65]{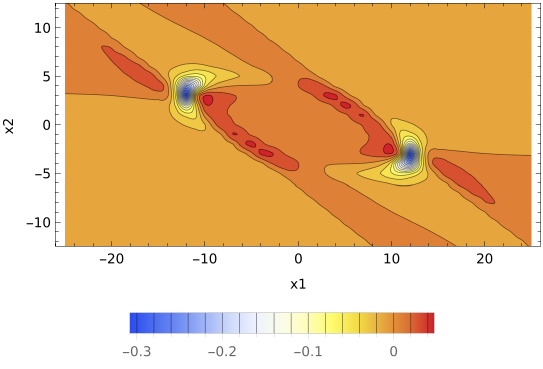}
\end{center}
\caption{The monoclonic case. The function $v^2(x,\,t)$, defined by \eqref{v_ist_seker}, corresponding to the second solution of the spectral problem \eqref{spec_prob}.  On the left we have the 3D plot of the function, on the right we have the level lines of the same function. The value $\gamma=0.2$, the time is $t=60$. \label{pic_monoclinic_v2_gamma_02} }
\end{figure}
\begin{figure}[!h]
\begin{center}
\includegraphics[scale=0.65]{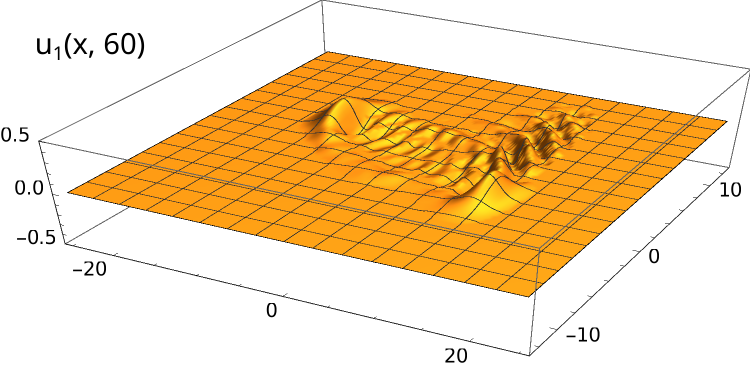}
\includegraphics[scale=0.65]{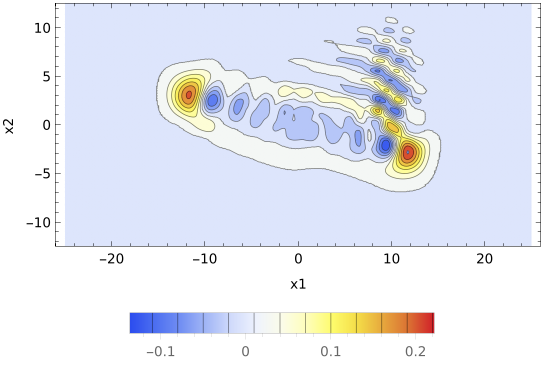}
\end{center}
\caption{The first component of the solution of the homogenized problem \eqref{u0_hom_eq} for the monoclinic case. On the left we have the 3D plot of the wave, on the right we have the level lines of the same wave. The value $\gamma=0.2$, the time is $t=60$. \label{pic_wave_monoclinic_U1_gamma_02} }
\end{figure}
\begin{figure}[!h]
\begin{center}
\includegraphics[scale=0.65]{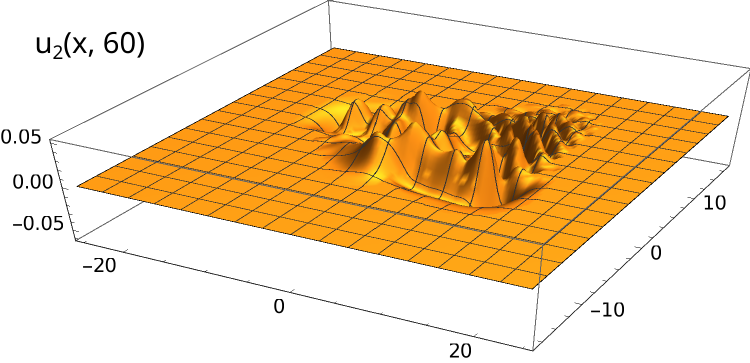}
\includegraphics[scale=0.65]{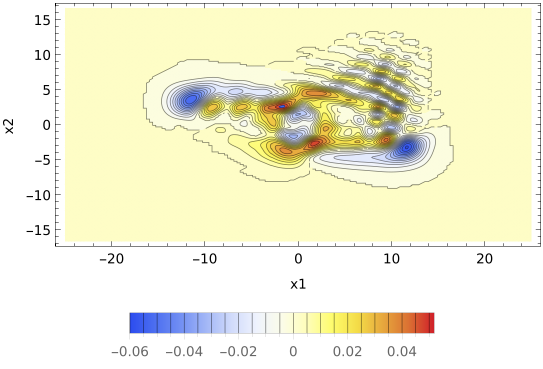}
\end{center}
\caption{The second component of the solution of the homogenized problem \eqref{u0_hom_eq} for the monoclinic case. On the left we have the 3D plot of the wave, on the right we have the level lines of the same wave. The value $\gamma=0.2$, the time is $t=60$. \label{pic_wave_monoclinic_U2_gamma_02} }
\end{figure}

\section{Discussion of results and conclusion.}
\label{conclus}

\paragraph{Interpretation of the effective coefficients.}
Table~\ref{tbl3_hom_tensor_gamma} shows a systematic and physically meaningful transition between the two constituent responses as the volume fraction $\gamma$ of material A increases. The longitudinal coefficient $\overline C_{1111}$ decreases monotonically from $63.1453$ GPa at $\gamma=0.2$ to $44.2578$ GPa at $\gamma=0.8$, i.e. by approximately $29.9\%$. This trend is consistent with the fact that material A is less stiff than material B in the $x_1$ direction. In contrast, $\overline C_{2222}$ and $\overline C_{1212}$ increase by approximately $68.7\%$ and $62.8\%$, respectively, over the same interval, reflecting the larger transverse and shear moduli of material A. The coupling coefficient $\overline C_{1122}$ also increases monotonically, from $2.37209$ to $3.46736$ GPa. Thus, changing $\gamma$ does not produce a uniform stiffening or softening: it redistributes the effective stiffness in a strongly direction-dependent manner.

The density varies linearly with $\gamma$, from $1498.4$ to $1793.6$ kg/m$^3$ in the tabulated range, as expected from the cell average $\langle\rho\rangle$. Consequently, the changes in the characteristic wave speeds are governed by the competition between effective stiffness and inertia. In the $x_1$-dominated response the simultaneous decrease of $\overline C_{1111}$ and increase of $\langle\rho\rangle$ tends to reduce the corresponding propagation speed as $\gamma$ grows, whereas the increase of $\overline C_{2222}$ and $\overline C_{1212}$ can compensate, or even reverse, this tendency for propagation and polarization involving the transverse and shear responses. This observation anticipates the angular dependence obtained from the spectral problem and emphasizes that the volume fraction acts as a genuine parameter for tuning anisotropic wave propagation.

The rotated example demonstrates that homogenization and material orientation interact non-trivially. The rotation itself is introduced through Eq.~\eqref{monoclinic_rotation}, while its direct effect on the homogenized elastic coefficients is quantified in Table~\ref{tbl_monoclinic}. For $\theta=\pi/12$, the effective tensor acquires the coupling coefficients $\overline C_{1112}$ and $\overline C_{2212}$ (reported in Voigt notation as $c_{16}$ and $c_{26}$), which are identically zero in the orthorhombic configuration. Their negative values are therefore a direct numerical measure of the loss of alignment between the material symmetry axes and the fixed coordinates. As $\gamma$ increases from $0.2$ to $0.8$, the magnitude of $c_{16}$ decreases from $13.109$ to $7.74734$ GPa (about $40.9\%$), while the magnitude of $c_{26}$ decreases from $1.27751$ to $0.939101$ GPa (about $26.5\%$). Thus, for this fixed rotation angle, increasing the fraction of material A weakens the rotation-induced coupling in the homogenized medium.

The spectral manifestation of this rotation-induced anisotropy is shown in Fig.~\ref{pic_monoclinic_eigenvalues}, where the angular dependence of the eigenvalues $\lambda^{1,2}(\phi)$ and the associated eigenvector fields are displayed for $\gamma=0.2$. In particular, the departure from the orthorhombic response can be appreciated by comparing Fig.~\ref{pic_monoclinic_eigenvalues} with Fig.~\ref{pic_orthorombic_eigenvalues}: the rotation modifies the directional structure of the spectral branches and the orientation of the polarization fields because the normal and shear components are now coupled through $c_{16}$ and $c_{26}$. Therefore, Fig.~\ref{pic_monoclinic_eigenvalues} should be read as the spectral counterpart of the effective coefficients listed in Table~\ref{tbl_monoclinic}, rather than as an independent numerical observation.

The diagonal and normal-coupling coefficients exhibit a different sensitivity. Over the same range, $c_{11}$ decreases by about $28.5\%$, whereas $c_{22}$ increases by about $63.1\%$. By contrast, $c_{66}$ changes only slightly between the endpoints (about $1.8\%$), although its dependence on $\gamma$ is weakly non-monotone. This contrast is useful because it shows that the rotated laminate can substantially modify longitudinal/transverse stiffness and normal--shear coupling without producing a comparably large net change in the effective shear coefficient. The volume fraction and the rotation angle therefore control distinct, but coupled, aspects of the effective anisotropy.

\paragraph{Spectral anisotropy and modal structure.}
Figure~\ref{pic_orthorombic_eigenvalues} provides the spectral counterpart of the effective coefficients. The two branches $\lambda^1(\phi)$ and $\lambda^2(\phi)$ remain separated and vary with the propagation angle, which is the expected signature of an anisotropic homogenized medium. The associated vector fields $\bchi_0^s(p)$ are not merely auxiliary quantities: they determine the polarization carried by each branch and, through formula~\eqref{v_ist_seker}, the amount of the initial disturbance projected onto each propagating mode. The symmetry visible in the orthorhombic plots is inherited from the absence of the coupling coefficients $c_{16}$ and $c_{26}$. Hence, the numerical spectral calculation confirms that the homogenization procedure preserves the principal material symmetries while replacing the rapidly layered medium by a direction-dependent effective propagation law.

Comparison of Figures~\ref{pic_orthorombic_eigenvalues} and~\ref{pic_monoclinic_eigenvalues} isolates the effect of the rotation while keeping the same constituents and the same volume fraction $\gamma=0.2$. The monoclinic coefficients introduce additional normal--shear coupling into $M(p)$; consequently, the angular dependence of the eigenvalues and the orientation of the polarization vectors are no longer constrained by the same reflection symmetries as in the orthorhombic case. In physical terms, the preferred directions of propagation and polarization rotate with respect to the coordinate axes. This is precisely the type of information that is difficult to infer from the entries of the effective tensor alone but becomes transparent in the spectral representation.

\paragraph{Wave-field reconstruction.}
Figures~\ref{pic_v1_gamma_02}--\ref{pic_wave_U2_gamma_02} clarify how the spectral decomposition is converted into the physical displacement field. Although the initial perturbation is polarized entirely along the first coordinate direction, both modal amplitudes $v^1$ and $v^2$ participate in the reconstruction because the eigenvectors depend on the propagation direction. The resulting field therefore cannot, in general, be interpreted as a single scalar pulse transported with one velocity. Instead, it is a superposition of two anisotropic elastic modes whose angularly dependent speeds and polarizations reshape the initially localized disturbance.

A particularly relevant feature is the appearance of a non-zero second displacement component despite the choice $\bfv(x)=V(x)(1,0)^T$. This is not generated by an imposed second-component initial displacement; it follows from the directional modal projection and the off-diagonal terms of the effective acoustic matrix $M(p)$. In this sense, the numerical solution illustrates an important consequence of anisotropy: propagation can redistribute displacement between components even when the excitation is initially aligned with a coordinate axis. The agreement between this reconstructed modal picture and the direct numerical solution of the homogenized equation provides a consistency check for the operator-separation representation used in the analysis.

The monoclinic wave fields in Figures~\ref{pic_monoclinic_v1_gamma_02}--\ref{pic_wave_monoclinic_U2_gamma_02} should be read together with the rotated spectral data rather than as a second independent example. The non-zero $c_{16}$ and $c_{26}$ terms alter both the angular phase speeds and the eigenvector orientations; therefore they modify simultaneously the geometry of the propagating fronts and the partition of the response between $u_1^0$ and $u_2^0$. Since the initial data, volume fraction and observation time are unchanged, differences with the orthorhombic fields can be attributed to the rotation-induced constitutive coupling rather than to a different excitation.

This comparison highlights the main numerical message of the section. The effective tensor obtained by the operator separation of variables is not only a static set of averaged coefficients: through the spectral problem it retains the directional information needed to predict the macroscopic dynamics of a localized disturbance. The examples show, in successive steps, how constituent properties and volume fraction determine the effective moduli, how those moduli determine the two dispersionless acoustic branches of the leading-order homogenized model, and how the branches and their polarizations reconstruct the observed displacement field. Rotation then provides a controlled symmetry-breaking test, demonstrating that the same framework captures the resulting coupling without changing the structure of the homogenized Cauchy problem.

\section{Operator separation of variables}
\label{sec_Sep_Var}

\subsection{Separation of the variables.}
Let us rewrite the problem  (\ref{eq1}), (\ref{init_cond}), using \eqref{stress_tensor}, as an equation for $\bu(x,\,t)$:
\begin{equation}
\label{main_eq}
\nabla_x\cdot (C(\frac{x_1}{\varepsilon}): \nabla_x\otimes \bu^\varepsilon)-\rho(\frac{x_1}{\varepsilon}) \bu^\varepsilon_{tt}=0,\quad
\bu^\varepsilon|_{t=0}=\bfv(x),\, \left.\frac{\partial}{\partial t}\bu^\varepsilon\right|_{t=0}=\bg(x).
\end{equation}
We are looking for the solution in the form
$$
\bu=\bFi(x,\,\frac{x_1}{\varepsilon},\,t;\, \varepsilon),
$$
where function $\bFi(x,\,y_1,\,t;\,\varepsilon)$ is the periodic function of the variable $y_1$ with period $l_1$. The variables $x$ are called slow variables and $y_1$ is called fast variable.

For the function $\bFi(x,\,y_1,\,t;\,\varepsilon)$ we get the following problem
\begin{gather}
\label{main_Fi_eq}
\mathcal{L}^\varepsilon(y_1)\bFi(x,\,y_1,\,t;\,\varepsilon)-\rho(y_1)\bFi_{tt}(x,\,y_1,\,t;\,\varepsilon)=0,\\
\label{init_cond_Fi}
\bFi(x,\,y_1,\,0;\,\varepsilon)=\bfv(x),\quad \bFi_t(x,\,y_1,\,0;\,\varepsilon)=\bg(x),\\
\label{interface_1_Fi}
\bFi^{-}(x,\,\gamma l_1,\,t;\,\varepsilon)=\bFi^{+}(x,\,\gamma l_1,\,t;\,\varepsilon),\\
\label{interface_2_Fi}
\Bigl[C^{1}(\gamma l_1):(\nabla_x \otimes \bFi)+\frac{1}{\varepsilon}(\mC(\gamma l_1)\cdot\bFi')\Bigr]^{-}
=\Bigl[C^{1}(\gamma l_1):(\nabla_x \otimes \bFi)+\frac{1}{\varepsilon}(\mC(\gamma l_1)\cdot\bFi')\Bigr]^{+}.
\end{gather}
The operator $\mathcal{L}^\varepsilon(y_1)$ is obtained by the direct substitution of the function $\bFi$ and the fact, that it depends only on $y_1$. This operator has the form
\begin{gather}
\label{oper_L_eps}
\mathcal{L}^\varepsilon(y_1)\bFi=\nabla_x\cdot (C(y_1): (\nabla_x\otimes \bFi)+\frac{1}{\varepsilon}\Bigl(C^1(y_1):(\nabla_x\otimes \bFi)\Bigr)'+\\
\nonumber
+\frac{1}{\varepsilon}C_1(y_1):(\nabla_x\otimes \bFi')+\frac{1}{\varepsilon^2}\Bigl(\mC(y_1)\bFi'\Bigr)'.
\end{gather}
Here the symbol ``$\prime$'' denotes derivative with respect to the $y_1$. 

Tensors $C_1(y_1)$, $C^1(y_1)$ and the matrix $\mC(y_1)$ are defined in \eqref{tens_C1_Cup1} and \eqref{matr_C}. Double contraction for the tensor $C_1(y_1)$ is by the last two indicies
$$
C_1(y_1):(\nabla_x\otimes \bFi')=C_{ijk1}(y_1)\frac{\partial \Phi_k'}{\partial x_j},\quad
C^1(y_1):(\nabla_x\otimes \bFi')=C_{i1kl}(y_1)\frac{\partial \Phi_k'}{\partial x_l}.
$$

We also introduced the values of the functions on the both sides of the interface
$$
\bFi^{\pm}(x,\,\gamma l_1,\,t;\,\varepsilon)=\bFi(x_1,\,\gamma l_1\pm 0,\,x_2,\,t;\,\varepsilon).
$$

The condition (\ref{interface_1_Fi}) comes from the continuity condition $\llbracket \bu^\varepsilon \rrbracket=0$ on the interface.
The condition (\ref{interface_2_Fi}) comes from the second condition on the interface --- continuity of the normal component of the stresses $\llbracket \bsigmaE\cdot \bn\rrbracket=0$.  In the present case the outward normal vector $\bn=(1,\,0)$. This gives 
$$
\bsigmaE\cdot \bn=\sigma^\varepsilon_{ij}n_j=\sigma^\varepsilon_{i1}=C_{i1kl}(\frac{x_1}{\varepsilon})\frac{\partial u_k^\epsilon}{\partial x_l}.
$$

\subsection{Fourier transform.}
Equations (\ref{main_Fi_eq})-(\ref{interface_2_Fi}) are the equations with constant coefficients with respect to the variable $x$ and thus we can pass to the Fourier transform in these equations with respect to the variable $x$. For the Fourier transform \eqref{Fourier_trans} of the derivative, we use the property
$$
F[f'(x)](p)=ip F[f(x)](p).
$$
We denote $\bp=(p_1,\,p_2)$ the vector of the dual variables to $x$, and we use the symbol ``$\hat{\quad}$'' to distinguish the Fourier transform of the function. 

After the implementation of the Fourier transform to the equation \eqref{oper_L_eps},   we have the equation
$$
\mathcal{L}^{\varepsilon}(y_1,\,p)\hat{\bFi}=-\bp\cdot (C(y_1): (\bp\otimes \hat{\bFi})+\frac{i}{\varepsilon}\Bigl(C^1(y_1):(\bp\otimes \hat{\bFi})\Bigr)'
+\frac{i}{\varepsilon}C_1(y_1):(\bp\otimes \hat{\bFi}')+
\frac{1}{\varepsilon^2}\Bigl(\mC(y_1)\hat{\bFi}'\Bigr)'.
$$

This gives the following system of equations
\begin{gather*}
\mathcal{L}^\varepsilon(y_1,\,p)\hat{\bFi}-\rho(y_1)\hat{\bFi}_{tt}=0,\\
\hat{\bFi}|_{t=0}=\hat{\bfv}(\bp),\quad \hat{\bFi}_t|_{t=0}=\hat{\bg}(\bp),\\
\hat{\bFi}^{-}=\hat{\bFi}^{+},\quad
\Bigl[i C^{1}(\gamma l_1):(\bp \otimes \hat{\bFi})+\frac{1}{\varepsilon}(\mC(\gamma l_1)\cdot\hat{\bFi}')\Bigr]^{-}=
\Bigl[i C^{1}(\gamma l_1):(\bp \otimes \hat{\bFi})+\frac{1}{\varepsilon}(\mC(\gamma l_1)\cdot\hat{\bFi}')\Bigr]^{+}.
\end{gather*}

We want to separate the fast variables $y_1$ and variables $p$. We cannot do that exactly, but we can look for the function $\hat{\bFi}$ in the following form of the product
\begin{equation}
\label{Phi_v_sep_var}
\hat{\bFi}=\bchi(y_1,\,p) \hat{v}(p,\,t),\quad 
\hat{v}(p,\,t)=A(p)e^{i t \omega(p)},
\end{equation}
for some function $A(p)$ and phase $\omega(p)$.

Note that the vector-field $\bchi(y_1,\,p)$, the function $\hat{v}(p,\,t)$, as well as $A(p)$ and $\omega(p)$, depend on the small parameter $\varepsilon$. We do not show explicitly this dependence to avoid cumbersome formulas.

Using analogy between variable $x$ and the dual variable $p$ we can identify function $\bchi(y_1,\,p)$ as symbol of a pseudo-differrential operator $\bchi(y_1,\,-i\nabla_x)$ which acts on the function $v(x,\,t)$ (definition we provide in Appendix \ref{PsiDO}). The function $v(x,\,t)$ (it is the inverse Fourier transform of $\hat{v}(p,\,t)$), following (\ref{Phi_v_sep_var}), satisfies the equation with constant coefficients
$$
-v_{tt}(x,\,t)=\omega^2(-i\nabla_x)v(x,\,t).
$$

Now we need to determine the function $\bchi(y_1,\,p)$ and phase $\omega(p)$. First we note, that operator $\mathcal{L}^\varepsilon(y_1,\,p)$ is the differential operator with respect to the variable $y_1$. 
Using definition (\ref{Phi_v_sep_var}) of the function $\hat{v}(p,\,t)$, the substitution and it's cancellation give the following problem 
\begin{gather}
\label{main_eq_chi}
\mathcal{L}^{\varepsilon}(y_1,\,p)\bchi(y_1,\,p)+\rho(y_1)\omega^2(p)\bchi(y_1,\,p)=0,\\
\label{period_cond_chi}
\bchi(0,\,p)=\bchi(l_1,\,p),\quad \bchi'(0,\,p)=\bchi'(l_1,\,p),\\
\label{interface_1_chi}
\bchi^{-}(\gamma l_1,\,p)=\bchi^{+}(\gamma l_1,\,p),\\
\label{interface_2_chi}
\Bigl[i C^{1}(\gamma l_1):(p \otimes \bchi)+\frac{1}{\varepsilon}(\mC(\gamma l_1)\cdot\bchi')\Bigr]^{-}
=\Bigl[i C^{1}(\gamma l_1):(p \otimes \bchi)+\frac{1}{\varepsilon}(\mC(\gamma l_1)\cdot\bchi')\Bigr]^{+}.
\end{gather}

We can see that the problem (\ref{main_eq_chi})-(\ref{interface_2_chi}) is the spectral problem with respect to the fast variable $y_1$ with periodic coefficients. Here the vector-field $\bchi(y_1,\,p)$ is the eigenfunction and the phase $\omega^2(p)$ is the eigenvalue. The variable $p$ here is the parameter.

\begin{definition}
Let us define the following functional space of periodic vector-functions with continuity condition on the interface $y_1=\gamma l_1$
$$
H^1_{per, I}=\{ \bu(y_1)\in H^1[0,\,l_1]\times H^1[0,\,l_1]\,:\, \bu(0)=\bu(l_1),\,\bu'(0)=\bu'(l_1),\,\bu^{-}(\gamma l_1)=\bu^{+}(\gamma l_1)\}.
$$
with the following inner product
\begin{equation}
\label{H_inner_prod}
(\bu,\,\bv)_{H^1_{per, I}}=\frac{1}{l_1}\left(\int\limits_{0}^{\gamma l_1}\langle \bu,\,\overline{\bv}\rangle dy_1+\int\limits_{\gamma l_1}^{l_1}\langle \bu,\,\overline{\bv}\rangle dy_1\right),
\end{equation}
where the line denotes the complex conjugation.
\end{definition}

Let us multiply the equation (\ref{main_eq_chi}) by $\varepsilon^2$ and then integrate it by parts after multiplication by the test function $\bpsi\in H^1_{per, I}$. We have the following
\begin{gather}
\nonumber
(\mC(y_1)\bchi',\,\bpsi')+i\varepsilon\Bigl((C^1(y_1):(\bp\otimes \bchi),\,\bpsi')-(C_1(y_1):(\bp\otimes \bchi'),\,\bpsi)\Bigr)=\\
\label{main_spec_eq_weak}
=\varepsilon^2\Bigl(\rho(y_1)\omega^2(p)\bchi-\bp\cdot (C(y_1): (\bp\otimes \bchi),\,\bpsi\Bigr).
\end{gather}

We impose the normalization condition on the $\chi$:
\begin{equation}
\label{norm_chi}
\|\bchi\|_{H^1_{per, I}}=1.
\end{equation}

\begin{definition}
We call the vector $\bchi(y_1,\,p)$ the weak solution of the spectral problem (\ref{main_eq_chi})-(\ref{interface_2_chi}), if $\bchi(y_1,\,\cdot)\in H^1_{per, I}$, satisfies  \eqref{main_spec_eq_weak} for any $\bpsi(y_1)\in H^1_{per, I}$ and the normalization condition \eqref{norm_chi}.
\end{definition}

\section{Asymptotic Expansion. Proof of the theorem \ref{thm_hom_eq}.}
\label{sec_hom_thm_proof}

We do not discuss the existence of the solution of \eqref{main_spec_eq_weak}, we assume it exists. We construct the asymptotic expansion of this solution in small parameter $\varepsilon$. This leads to the so-called the formal asymptotic solution.
Let us suppose that $\bchi(y_1,\,p)$ and $\omega(p)$ can be presented as the following expansion
\begin{equation}
\label{chi_expans}
\bchi(y_1,\,p)=\bchi_0(y_1,\,p)-i\varepsilon \bchi_1(y_1,\,p)-\varepsilon^2\bchi_2(y_1,\,p)+\ldots,\quad 
\omega^2(p)=\omega_0^2(p)+\varepsilon\omega_1(p)+\ldots.
\end{equation}
The normalization condition \eqref{norm_chi} leads to the following relations
\begin{equation}
\label{norm_chi0_chi1}
\|\bchi_0(y_1,\,p)\|_{H^1_{per, I}}=1,\quad Im(\bchi_0,\,\bchi_1)=0,\quad
\|\bchi_1(y_1,\,p)\|^2_{H^1_{per, I}}-2Re\, (\bchi_0,\,\bchi_2)=0.
\end{equation}

\begin{theorem}[{\bf Asymptotic expansion}]
\label{thm_as}
Let $((\omega^s_0(p))^2,\,\bchi^s_0(p))$, $s=1,\,2$, be the solution of the spectral problem (\ref{spec_prob}) such that $|\bchi^s_0(p)|=1$. Let the functions $v^s(x,\,t)$ be the solution of the problems
\begin{gather}
\label{homogen_eq}
-\frac{\partial^2}{\partial t^2}v^s(x,\,t)=\left(\omega_0^s(-i\nabla_x)\right)^2v^s(x,\,t),\\
\nonumber
v^s(x,\,0)=\langle\bchi_0^s(-i\nabla_x),\,\bfv(x)\rangle,\quad \frac{\partial}{\partial t}v^s(x,\,0)=\langle\bchi_0^s(-i\nabla_x),\,\bg(x)\rangle.
\end{gather}
Then the function
\begin{equation}
\label{u_as_sol}
\bu_{as}(x,\,t)=\sum\limits_{s}\Bigl(\bchi_0^s(-i\nabla_x)-i\varepsilon \bchi_1(\frac{x_1}{\varepsilon},\,-i\nabla_x)-\varepsilon^2\bchi_2(\frac{x_1}{\varepsilon},\,-i\nabla_x)\Bigr)v^s(x,\,t)
\end{equation}
satisfies the  equation (\ref{main_eq}) and initial conditions \eqref{init_cond} with the residue $O(\varepsilon)$.

The symbols of the operators $\bchi_{1,\,2}^s(y_1,\,p)$ are defined as follows
$$
\bchi^s_1(y_1,\,p)=\left(\int\limits_{0}^{y_1}\mathcal{A}(\eta)d\eta\right):(p\otimes\bchi^s_0),
$$
where the tensor $\mathcal{A}(y_1)$ is defined in \eqref{tensor_A_big}.
The vectors $\bchi_2^s(y_1,\,p)$ are the solution of the problem
\begin{gather}
\label{weak_prob_chi2}
(\mC(y_1)(\bchi^s_2)^{\prime}(y_1),\,\psi'(y_1))+(C^1(y_1):(\bp\otimes \bchi^s_1(y_1)),\,\psi'(y_1))=\\
\nonumber
-(C_1(y_1):(\bp\otimes (\bchi^s_1)'),\,\psi(y_1))-\Bigl(\rho(y_1)\omega_0^2(p)\bchi^s_0-\bp\cdot (C(y_1): (\bp\otimes \bchi^s_0),\,\psi(y_1)\Bigr),\forall\,\psi(y_1)\in H^1_{per,I}.
\end{gather}

\end{theorem}

\begin{proof}
The proof of this theorem is given in two parts. The first part is about the calculations of the terms $\bchi_{0,1,2}(y_1,\,p)$ in the expansion \eqref{chi_expans}. The second part is about verifying the residue of the substitution of \eqref{u_as_sol} into the original equation.

\paragraph{Calculation of the terms of the expansion.} Substitution of the expansion \eqref{chi_expans} into the equation \eqref{main_spec_eq_weak} leads to the sequence of the problems to determine the terms of the expansion. 

The function $\bchi_0$ satisfies the following problem
$$
(\mathcal{C}(y_1)\bchi_0'(y_1),\,\psi'(y_1))_{H^1_{per, I}}=0,\quad \forall\,\psi(y_1)\in H^1_{per, I}.
$$
The solution of this problem is the combination of two linearly independent vectors $\bchi^{1,2}_0$ which do not depend on $y_1$. 

For the function $\bchi_1$ we have the following equation
\begin{equation}
\label{eq_weak_chi1}
\Bigl(\mC(y_1)\bchi_1'(y_1)-C^1(y_1):(\bp\otimes \bchi_0),\,\psi'(y_1)\Bigr)=0,\quad \forall\,\psi(y_1)\in H^1_{per,I}.
\end{equation}
Since $\psi(y_1)\in H^1_{per,I}$, we have that derivative $\psi'(y_1)$ has the zero mean value. Because (\ref{eq_weak_chi1}) holds for any $\psi(y_1)$, we obtain that
$$
\mC(y_1)\bchi_1'(y_1)-C^1(y_1):(\bp\otimes \bchi_0)=\xi.
$$
The vector $\xi$ does not depend on $y_1$, it is convenient to present this vector in the following form
$$
\xi=A:(p\otimes \bchi_0),
$$
where $A=A_{ikl}$ is the some tensor.

In this case we have
$$
\bchi_1'(y_1)=\mathcal{C}^{-1}(y_1)\Bigl(A+C^1(y_1)\Bigr):(p\otimes\bchi_0).
$$
The tensor $A$ can be found from the condition of the periodicity of the $\bchi_1(y_1)$ such that it has the zero mean $\langle \bchi_1'(y_1) \rangle=0$. This gives the equation
$$
\langle \bchi_1'(y_1)\rangle=\Bigl(\langle\mathcal{C}^{-1}(y_1)\rangle A+\langle \mathcal{C}^{-1}(y_1) C^1(y_1)\rangle\Bigr):(p\otimes\bchi_0)=0.
$$
From this equation we determine the tensor $A$
$$
A=-\langle \mathcal{C}^{-1}(y_1)\rangle^{-1} \langle \mathcal{C}^{-1}(y_1) C^1(y_1)\rangle\Leftrightarrow 
A_{ikl}=-\Bigl(\langle \mathcal{C}^{-1}(y_1)\rangle^{-1}\Bigr)_{im} \langle \Bigl(\mathcal{C}^{-1}(y_1)\Bigr)_{mr} \Bigl(C^1(y_1)\Bigr)_{rkl}\rangle
$$

This leads to the  representation of the $\bchi_1'(y_1)$
$$
\bchi_1'(y_1)=\mathcal{A}(y_1):(p\otimes\bchi_0),
$$
where tensor $\mathcal{A}(y_1)$ is defined in \eqref{tensor_A_big}.

Finally for the function $\bchi_2$ we have the equation \eqref{weak_prob_chi2}. 
The solvalibility condition of this equation is given in the appendix in \cite{bakhvalov_homogenisation_1989}, and this condition is the orthogonality of the right-hand side of this equation to the constant vectors with respect to the variable $y_1$. Thus we have the following equation for the $\bchi_0$
$$
\frac{1}{l_1}\int\limits_{0}^{l_1}\rho(y_1)\omega_0^2(p)\bchi_0dy_1=
\frac{1}{l_1}\int\limits_{0}^{l_1}\bp\cdot (C(y_1): (\bp\otimes \bchi_0)dy_1-\frac{1}{l_1}\int\limits_{0}^{l_1}C_1(y_1):(\bp\otimes \bchi_1')dy_1.
$$

Now let us simplify the tensor terms in this expression. First of all, we have the equation
\begin{gather*}
[p\otimes\bchi_1'(y_1)]_{jk}=\Biggl[p\otimes \Bigl(\mathcal{A}(y_1):(p\otimes \bchi_0)\Bigr)\Biggr]_{jk}=\\
=\Bigl[p\otimes \Bigl(\mathcal{A}_{kmn}(y_1)p_m (\bchi_0)_n\Bigr)\Biggr]_{jk}=p_j\mathcal{A}_{kmn}(y_1)p_m (\bchi_0)_n.
\end{gather*}
After that we have
$$
\Bigl[C_1(y_1):(\bp\otimes \bchi_1')\Bigl]_{i}=C_{ijk1}(y_1)p_j\mathcal{A}_{kmn}(y_1)p_m (\bchi_0)_n.
$$
Now we use the symmetry property of the basic tensor $C_{ijkl}(y_1)=C_{jikl}(y_1)$ and thus we have
$$
C_{ijk1}(y_1)p_j=p_jC_{jik1}(y_1)=p_i C_{ijk1}(y_1).
$$
In the last equality we changed the names of indices $i\leftrightarrow j$. Finally we have the following
$$
\Bigl[C_1(y_1):(\bp\otimes \bchi_1')\Bigl]_{j}=p_iC_{ijk1}(y_1)\mathcal{A}_{kmn}(y_1)p_m (\bchi_0)_n=
\Biggl[p\cdot C_1(y_1)\cdot\mathcal{A}(y_1):(p\otimes\bchi_0)\Biggr]_j.
$$

As a result we obtain the following spectral problem for $\bchi_0$ and $\omega_0^2(p)$
leads to the equation (\ref{spec_prob})
$$
\langle\rho(y_1)\rangle\omega_0^2(p)\bchi_0=p\cdot \Bigl(\langle C(y_1)\rangle-\langle C_1(y_1)\cdot\mathcal{A}(y_1) \rangle
\Bigr):(p\otimes\bchi_0).
$$
This formula leads to the definition of the homogenized tensor \eqref{C_hom}. 

\paragraph{Verification of the residue.}
Function $\bu_{as}(x,\,t)$ is of the form \eqref{Phi_v_sep_var} with separated fast and slow variables, thus the action of the operator in \eqref{main_eq} on this function is described by the operator $\mathcal{L}^\varepsilon(y_1)$ defined in \eqref{oper_L_eps}.

Let us suppose that $\bchi(y_1,\,p)$ be the solution to the spectral problem \eqref{main_spec_eq_weak} and $v(x,\,t)$ satisfies the equation \eqref{homogen_eq}, then the function $\bu(x,\,t)=\bchi(x_1/\varepsilon,\,-i\nabla_x) v(x,\,t)$ satisfies the initial equation \eqref{main_eq}. Here we have to take into account that $\bchi(x_1/\varepsilon,\,-i\nabla_x)$ is the pseudo-differential operator with standard quantization (the operator $-i\nabla$ acts first and then multiplication by $x_1/\varepsilon$), see  Appendix \ref{PsiDO}.

Substitution of the function $\bu(x,\,t)$ into the equation \eqref{main_eq} leads to the following equality
$$
\nabla_x\cdot (C(\frac{x_1}{\varepsilon}): \nabla_x\otimes \bu(x,\,t))=\left(\mathcal{L}^\varepsilon(y_1,\,\partial_{y_1},\,\nabla_x)\bchi(y_1,\,-i\nabla_x) v(x,\,t)\right)\Biggr|_{y_1=x_1/\varepsilon}.
$$

We have the composition of the differential operator $\mathcal{L}^\varepsilon(y_1)$ and pseudo-differential operator $\bchi(y_1,\,-i\nabla_x)$. Because both these operators do not depend on variable $x$, we can write the action in the simple Fourier form
$$
\mathcal{L}^\varepsilon(y_1,\,\partial_{y_1},\,\nabla_x)\bchi(y_1,\,-i\nabla_x) v(x,\,t)=F^{-1}[\mathcal{L}^\varepsilon(y_1,\,\partial_{y_1},\,p)\bchi(y_1,\,p) \hat{v}(p,\,t)].
$$
Because the vector $\bchi(y_1,\,p)$ is the eigenvector of the operator $\mathcal{L}^\varepsilon(y_1,\,\partial_{y_1},\,p)$, we have
$$
\mathcal{L}^\varepsilon(y_1,\,\partial_{y_1},\,\nabla_x)\bchi(y_1,\,-i\nabla_x) v(x,\,t)=F^{-1}[\rho(y_1)\bchi(y_1,\,p) \omega^2(p)\hat{v}(p,\,t)]=
\rho(y_1)\bchi(y_1,\,-i\nabla_x) \omega^2(-i\nabla_x) v(x,\,t),
$$
where we used once more the definition of the action of the pseudo-differential operator.

This leads to the following
$$
\nabla_x\cdot (C(\frac{x_1}{\varepsilon}): \nabla_x\otimes \bu(x,\,t))=\rho(\frac{x_1}{\varepsilon})\bchi(\frac{x_1}{\varepsilon},\,-i\nabla_x) \omega^2(-i\nabla_x) v(x,\,t).
$$

On the other hand, we have
$$
\rho(\frac{x_1}{\varepsilon})\bu_{tt}(x,\,t)=\rho(\frac{x_1}{\varepsilon})\bchi(\frac{x_1}{\varepsilon},\,-i\nabla_x)v_{tt}(x,\,t)=\rho(\frac{x_1}{\varepsilon})\bchi(\frac{x_1}{\varepsilon},\,-i\nabla_x) \omega^2(-i\nabla_x) v(x,\,t).
$$ 
This shows that the function $\bu(x,\,t)$ is the solution of the original equation \eqref{main_eq}. 

The expansion for $\bchi(y_1,\,p)$ given in \eqref{u_as_sol} is constructed in the way that it satisfies the equation \eqref{main_spec_eq_weak} with error $O(\varepsilon)$. Thus the function $\bu_{as}(x,\,t)$ satisfies the equation with order $O(\varepsilon)$.

For the initial conditions we have the correction of the first order $O(\varepsilon)$, which follows from the same expansion for the operator $\bchi(y_1,\,p)$. Since the normalization condition \eqref{norm_chi0_chi1} we have
$$
\bchi_0(p)\hat{v}(p,\,0)=\hat{\bfv}(p)\Rightarrow \hat{v}(p,\,0)=\bchi_0^T(p)\hat{\bfv}(p)\Rightarrow v(x,\,0)=\bchi_0^T(-i\nabla_x)\bfv(x).
$$
The same reasoning is valid for the initial velocity $v_t(x,\,0)=\bchi_0^T(-i\nabla_x)\bg(x)$.

\end{proof}

Using the Theorem \ref{thm_as} we can construct the correction term and obtain the expansion \eqref{u_as_sol_main}. The function $\bu_2(x,\,t)$ in this expansion is 
$$
\bu_2(x,\,t)=\sum\limits_{s=1}^{2}\bchi^s_2(x_1/\varepsilon,\,-i\nabla_x)v_s(x,\,t).
$$

The Theorem \ref{thm_as} concludes the proof of the Homogenization theorem \ref{thm_hom_eq}.

\appendix
\section{Proof of the Theorem \ref{thm_hom_tensor}}
\label{proof_thm_hom_tensor}

Recall that the homogenized tensor \eqref{C_hom}  has the form
$$
\overline{C}=\langle C\rangle-\langle C_1\mathcal{C}^{-1}C^1\rangle+\langle C_1\mathcal{C}^{-1}\rangle \langle \mathcal{C}^{-1}\rangle^{-1} \langle \mathcal{C}^{-1}C^1\rangle.
$$

Let us demonstrate that it satisfies the major symmetry property.  We will show it for each component of this tensor. It is obvious that $\langle C(y_1)\rangle$ satisfies the symmetry since it is just averaging of the original tensor.

Together with \eqref{C1_Cup1_symmetry},  \eqref{C1_T_permut} and symmetry of the matrix $(\mathcal{C}^{-1}(y_1))=(\mathcal{C}^{-1}(y_1))^T$, we have the symmetry for the second part of the tensor
\begin{gather*}
\Bigl(C_1(y_1)\mathcal{C}^{-1}(y_1)C^1(y_1)\Bigr)_{ijkl}=(C_1(y_1))_{ijm}(\mathcal{C}^{-1}(y_1))_{mp}(C^1(y_1))_{pkl}=\\
=(C^1(y_1))_{mij}(\mathcal{C}^{-1}(y_1))_{mp}(C_1(y_1))_{klp}=(C_1(y_1))_{klp}(\mathcal{C}^{-1}(y_1))_{pm}(C^1(y_1))_{mij}=\Bigl(C_1(y_1)\mathcal{C}^{-1}(y_1)C^1(y_1)\Bigr)_{klij}.
\end{gather*}

For the third term, we have 
\begin{gather*}
\Bigl(\langle C_1(y_1)\mathcal{C}^{-1}(y_1)\rangle\langle \mathcal{C}^{-1}(y_1)\rangle^{-1} \langle \mathcal{C}^{-1}(y_1) C^1(y_1) \rangle\Bigr)_{ijkl}=\\
=\langle (C_1(y_1))_{ijm}(\mathcal{C}^{-1}(y_1))_{mp}\rangle(\langle \mathcal{C}^{-1}(y_1)\rangle^{-1})_{pq} \langle (\mathcal{C}^{-1}(y_1))_{qr} (C^1(y_1))_{rkl}\rangle.
\end{gather*}
Now we have the following sequence of equalities
\begin{gather*}
\langle (C_1(y_1))_{ijm}(\mathcal{C}^{-1}(y_1))_{mp}\rangle=\langle (\mathcal{C}^{-1}(y_1))_{mp} (C^1(y_1))_{mij}\rangle=\langle (\mathcal{C}^{-1}(y_1))_{pm} (C^1(y_1))_{mij}\rangle.
\end{gather*}
Here we just changed order of multiplication of two numbers $(\mathcal{C}^{-1}(y_1))_{mp}$ and $(C^1(y_1))_{mij}$ and then we used the symmetry $(\mathcal{C}^{-1}(y_1))=(\mathcal{C}^{-1}(y_1))^T$.
Similarly we have
$$
\langle (\mathcal{C}^{-1}(y_1))_{qr} (C^1(y_1))_{rkl}\rangle=\langle (C_1(y_1))_{klr} (\mathcal{C}^{-1}(y_1))_{rq}\rangle.
$$

Taking into account the symmetry $(\langle \mathcal{C}^{-1}(y_1)\rangle^{-1})=(\langle \mathcal{C}^{-1}(y_1)\rangle^{-1})^T$ we arrive that 
$$
\Bigl(\langle C_1(y_1)\mathcal{C}^{-1}(y_1)\rangle\langle \mathcal{C}^{-1}(y_1)\rangle^{-1} \langle \mathcal{C}^{-1}(y_1) C^1(y_1) \rangle\Bigr)_{ijkl}=
\Bigl(\langle C_1(y_1)\mathcal{C}^{-1}(y_1)\rangle\langle \mathcal{C}^{-1}(y_1)\rangle^{-1} \langle \mathcal{C}^{-1}(y_1) C^1(y_1) \rangle\Bigr)_{klij}.
$$

The right and left minor symmetries of $\overline{C}$ are the corollary of the \eqref{C1_Cup1_symmetry}.

Now let us show that the tensor \eqref{C_hom} is positive definite, meaning
$$
\overline{C}\xi_{ij}\xi_{kl}\ge A |\xi_{mn}|^2>0, 
$$
where $A>0$ defined in \eqref{C_tens}.

We start from the inequality
\begin{equation}
\label{C_C_tens_ineq}
C(y_1)\ge C_1(y_1)\mC^{-1}(y_1)C^1(y_1)\Leftrightarrow (C(y_1)-C_1(y_1)\mC^{-1}(y_1)C^1(y_1))\xi_{ij}\xi_{kl}\ge 0.
\end{equation}

According to \eqref{C_tens}, the tensor $C(y_1)$ is positive definite, meaning that $C_{ijkl}(y_1)\eta_{ij}\eta_{kl}\ge A|\eta_{mn}|^2$, for any $\eta_{ij}$. Let us choose
$$
\eta_{ij}=\xi_{ij}+\frac{1}{2}(w_i\delta_{j1}+w_j\delta_{i1}),\quad \forall\,\xi_{ij},
$$
where $\delta_{mn}$ is the Kronecker $\delta$-symbol.
For that particular form of $\eta_{ij}$ we have
\begin{gather}
\nonumber
C(y_1)_{ijkl}\eta_{ij}\eta_{kl}=\\
\nonumber
=C(y_1)_{ijkl}\xi_{ij}\xi_{kl}+C(y_1)_{ijkl}\xi_{ij}(w_k\delta_{l1}+w_l\delta_{k1})+\frac{1}{4}C(y_1)_{ijkl}(w_i\delta_{j1}+w_i\delta_{j1})(w_k\delta_{l1}+w_l\delta_{k1})=\\
\label{C_eta_tensor}
=C(y_1)_{ijkl}\xi_{ij}\xi_{kl}+2C(y_1)_{ijkl}\xi_{ij}w_k\delta_{l1}+C(y_1)_{ijkl}w_i\delta_{j1}w_k\delta_{l1}.
\end{gather}
Here for the simplification of the second and the third terms in this expression we used the symmetry of the tensor $C(y_1)$:
$$
C(y_1)_{ijkl}\xi_{ij}w_k\delta_{l1}=C(y_1)_{ijlk}\xi_{ij}w_k\delta_{l1}=C(y_1)_{ijkl}\xi_{ij}w_l\delta_{k1}.
$$
The last equality is valid because of the summation over both indices $k$ and $l$.

Calculating the summation over $\delta$-symbols in \eqref{C_eta_tensor}, we arrive to the following
$$
C(y_1)_{ijkl}\eta_{ij}\eta_{kl}=C(y_1)_{ijkl}\xi_{ij}\xi_{kl}+2C(y_1)_{ijk1}\xi_{ij}w_k+C(y_1)_{i1k1}w_iw_k.
$$
We rewrite this equality in the compact form. Using \eqref{tens_C1_Cup1} and \eqref{matr_C}, we have
\begin{equation}
\label{C_eta_tensor_fin}
\eta^TC(y_1)\eta=\xi^TC(y_1)\xi+2\xi^TC_1(y_1)w +w^T\mC(y_1)w>0.
\end{equation}

The equation \eqref{C_eta_tensor_fin} is the quadratic form with respect to the vector $w$. The argument of the minimum of this form is the vector $w^*=-\mC^{-1}(y_1) f$, where we denoted $f=C^1(y_1)\xi$. Substitution of the $w^*$ leads to the following 
$$
\xi^TC(y_1)\xi+2f^Tw^*+(w^{*})^{T}\mC(y_1)w^*=\xi^TC(y_1)\xi-2f^T\mC^{-1}(y_1)f+f^T\mC^{-1}(y_1)f=\xi^TC(y_1)\xi-f^T\mC^{-1}(y_1)f>0.
$$
Taking into account the form of the vector $f$, we arrive to the \eqref{C_C_tens_ineq}
$$
\xi^TC(y_1)\xi>f^T\mC^{-1}(y_1)f=\xi^T C_1(y_1)\mC^{-1}(y_1)C^1(y_1)\xi,\quad \forall\,\xi.
$$

From the inequality \eqref{C_C_tens_ineq} we obtain 
$$
\overline{C}>\langle C_1\mathcal{C}^{-1}\rangle \langle \mathcal{C}^{-1}\rangle^{-1} \langle \mathcal{C}^{-1}C^1\rangle.
$$
Now we show that the tensor in the right-hand side is positive definite. We have
$$
\langle C_1\mathcal{C}^{-1}\rangle \langle \mathcal{C}^{-1}\rangle^{-1} \langle \mathcal{C}^{-1}C^1\rangle\xi_{ij}\xi_{kl}=
\langle (\mathcal{C}^{-1}f)^T\rangle \langle \mathcal{C}^{-1}\rangle^{-1} \langle \mathcal{C}^{-1}f\rangle=
v^T \langle \mathcal{C}^{-1}\rangle^{-1}v.
$$
where the vector $f=C^1(y_1)\xi$ is the same as before and the vector $v=\langle \mathcal{C}^{-1}f\rangle$.

The matrix $\mC(y_1)$ is positive-definite, see \eqref{matr_C_posit}, therefore the matrix inverse matrix $\mC^{-1}(y_1)$ is positive definite too.  Thus, the matrices $\langle \mC^{-1}\rangle$  and $\langle \mC^{-1}\rangle^{-1}$ are also positive definite, moreover 
$$
v^T\langle \mC^{-1}\rangle^{-1}v\ge A |v|^2,
$$
where $A>0$ is the same constant that is in the \eqref{matr_C_posit}. Finally we arrive to 
$$
\overline{C}>\langle C_1\mathcal{C}^{-1}\rangle \langle \mathcal{C}^{-1}\rangle^{-1} \langle \mathcal{C}^{-1}C^1\rangle\ge A>0.
$$
This proves the theorem.

\section{Basic notion of the pseudo-differential operators}
\label{PsiDO}

Here we provide some useful information about the pseudo-differential operators ($\Psi$DO) and describe our approach to analysis of them, based on \cite{hormander_analysis_2007, maslov_semi-classical_2001, martinez_introduction_2002, zworski_semiclassical_2012}.

Let $x\in\mathbb{R}^2$, then the differential operator $D(x)$ of order $m$ with variable coefficients is the following operator 
$$
D(x)u(x)=\sum\limits_{i, j=0}^{m}a_{ij}(x)(-i)^{i+j} \frac{\partial^{i+j}}{\partial x_1^i \partial x_2^j} u(x),
$$
Formally speaking the pseudo-differential $\Psi D(x)$ operator is the infinite series
$$
\Psi D(x)u(x)=\sum\limits_{i, j=0}^{+\infty}a_{ij}(x)(-i)^{i+j} \frac{\partial^{i+j}}{\partial x_1^i \partial x_2^j}u(x).
$$

Now let us give the notion of the symbol of the $\Psi$DO. The symbol of the operator is the function, which define our operator. The general idea is the following: to work with the symbols of the operators, which are the functions, and only in the end pass to the action on the function.

We start from the example of the constant coefficients. Let us assume that function $u(x)$ is an analytical function and let us define the action of the $\Psi$DO defined as sine 
$$
\sin\left(-i\frac{\partial}{\partial x_1}-i\frac{\partial}{\partial x_2}\right)u(x).
$$
We can formally write the Taylor expansion 
$$
\sin\left(-i\frac{\partial}{\partial x_1}-i\frac{\partial}{\partial x_2}\right)u(x)=\sum\limits_{n=1}^{+\infty}(-1)^{2n-1}\frac{(-i)^{2n-1}}{(2n-1)!}\left(\frac{\partial}{\partial x_1}+\frac{\partial}{\partial x_2}\right)^{2n-1}u(x).
$$
After that we define the action of this operator via Fourier transform \eqref{Fourier_trans}, using the following property
$$
F[(-i\frac{\partial}{\partial x_j})^ku(x)](p)=\frac{1}{2\pi}\int\limits_{\mathbb{R}^2}(-i\frac{\partial}{\partial x_j})^ku(x) e^{-ip\cdot x}dx=p_j^k \frac{1}{2\pi}\int\limits_{\mathbb{R}^2}u(x) e^{-ip\cdot x}dx=p_j^k \hat{u}(p).
$$
Thus we can write
$$
F[\sin\left(-i\frac{\partial}{\partial x_1}-i\frac{\partial}{\partial x_2}\right)u(x)](p)=\sum\limits_{n=1}^{+\infty}(-1)^{2n-1}\frac{1}{(2n-1)!}(p_1+p_2)^{2n-1}\hat{u}(p)=\sin(p_1+p_2)\hat{u}(p),
$$
and the inverse Fourier transform gives
$$
\sin\left(-i\frac{\partial}{\partial x_1}-i\frac{\partial}{\partial x_2}\right)u(x)=F^{-1}[\sin(p_1+p_2)\hat{u}(p)](x)=\frac{1}{2\pi}\int\limits_{\mathbb{R}^2}\sin(p_1+p_2)\hat{u}(p)e^{i p\cdot x}dp.
$$
Function $\sin(p_1+p_2)$ is called the symbol of the $\Psi$DO $\sin\left(-i\frac{\partial}{\partial x_1}-i\frac{\partial}{\partial x_2}\right)$.

In the case of the variable coefficients we need to take into account the non-commutative properties of the operator of the differentiation in $x$ and operator of multiplication by variable $x$. For example to the one symbol $O(x,\,p)=px$ we can assign two different operators
\begin{gather*}
O_1(x,\,-i\varepsilon\frac{d}{dx})u(x)=x(-i\frac{d}{dx}u(x))=-i xu'(x),\\ 
O_2(x,\,-i\frac{d}{dx})u(x)=-i\frac{d}{dx}(xu(x))=-i xu'(x)-i\varepsilon u(x).
\end{gather*}
In the operator $O_1$ the differentiation acts first and then multiplication by variable $x$. In the operator $O_2$ the action is reverse. 

To distinguish these operators we introduce the so-called quantization --- the order of the action of the operators of the differentiation and multiplication. We denote this order by the numbers above, for example
$$
O_1(x,\,-i\frac{d}{dx})=O(\stackrel{2}{x},\,\stackrel{1}{-i\frac{d}{dx}}).
$$

\begin{definition}
The action of the pseudo-differential operator $L(\stackrel{2}{x},\,\stackrel{1}{-i\nabla})$
on the function $u(x)$ is defined via Fourier transform
$$
L(\stackrel{2}{x},\,\stackrel{1}{-i\nabla})u(x)=F^{-1}[L(x,\,p)\hat{u}(p)](x).
$$

Function $L(x,\,p)$ is called the symbol of the $\Psi$DO $L(\stackrel{2}{x},\,\stackrel{1}{-i\nabla})$. 
We call this quantization rule the classical quantization rule.

\end{definition}

The classical quantization rule is similar to what was done in the case of the constant coefficients.

We also need the results of the composition of the $\Psi$DO $L(\stackrel{2}{x},\,\stackrel{1}{-i\nabla})$ with differential operator. Let us calculate
$$
u_2(x)=-i \frac{\partial}{\partial x_j} u_1(x),\quad u_1(x)=L(\stackrel{2}{x},\,\stackrel{1}{-i\nabla}) u(x).
$$
We use the rule of the differentiation of the product, and thus we have
$$
u_2(x)=-i L_{x_j}(\stackrel{2}{x},\,\stackrel{1}{-i\nabla})u(x)+L(\stackrel{2}{x},\,\stackrel{1}{-i\nabla}) u_{x_j}(x)
$$

Thus the composition of the differential operator and of the $\Psi$DO is also a pseudo-differential operator with the symbol
$$
L(x,\,p)p_j-i \frac{\partial}{\partial x_j}L(x,\,p).
$$

\bibliographystyle{unsrt}      % or unsrt, alpha, ieeetr, abbrv, etc.
\bibliography{ref}

\end{document}